\documentclass[reqno,a4paper,12pt]{amsart} 
\usepackage[letterpaper, margin=1.00in]{geometry}

\usepackage{amscd}
\usepackage{mathrsfs}
\usepackage[IL2]{fontenc}
\usepackage{mathtools}
\usepackage{comment} % \begin{comment}\end{comment} to comment stuff that may be uncommented later on

\usepackage{color}
\usepackage{amssymb,amsmath,amsthm,amsfonts,latexsym}
\usepackage[utf8x]{inputenc}
\usepackage{bbm} %this is so we can use mathbb with numbers, but using the command bbm

\usepackage{graphicx}
\usepackage[textsize=tiny, textwidth=20mm]{todonotes}
\usepackage{amsmath,amsfonts,amssymb,amsthm} 
\usepackage{tikz}
\usepackage{hyperref}
\usepackage{thmtools}
\usepackage{thm-restate}
\newtheorem{theorem}{Theorem}[section]

\newtheorem*{claim*}{Claim}
\newtheorem*{theorem*}{Theorem}
\newtheorem*{definition*}{Definition}

\newtheorem{lemma}[theorem]{Lemma}
\newtheorem{remark}[theorem]{Remark}
\newtheorem{claim}[theorem]{Claim}

\newtheorem{proposition}[theorem]{Proposition}

\newtheorem{example}[theorem]{Example}
\newtheorem{definition}[theorem]{Definition}

\newcommand{\comm}[1]{}
\usepackage[capitalize, nameinlink]{cleveref}
\crefname{theorem}{Theorem}{Theorems}
\crefname{proposition}{Proposition}{Propositions}
\crefname{observation}{Observation}{Observations}
\crefname{lemma}{Lemma}{Lemmas}
\crefname{claim}{Claim}{Claims}
\crefname{problem}{Problem}{Problems}
\crefname{conjecture}{Conjecture}{Conjectures}
\crefname{question}{Question}{Questions}
\crefname{example}{Example}{Examples}
\crefname{fact}{Fact}{Facts}

\newcommand{\one}{{\bf 1}}
\newcommand{\Z}{\mathbb{Z}}
\newcommand{\N}{\mathbb{N}}
\newcommand{\T}{\mathbb{T}}
\newcommand{\Q}{\mathbb{Q}}
\newcommand{\R}{\mathbb{R}}
\newcommand{\id}{\mathrm{id}}
\renewcommand{\ae}{{\mathrm{a.e.}}}

\def\leukfrac#1/#2{\leavevmode
               \kern.1em
                \raise.9ex\hbox{\the\scriptfont0 ${}_#1$}
                \hskip -1pt\kern-.1em
                /\kern-.15em\lower.10ex\hbox{\the\scriptfont0 ${}_#2$}}

\makeatletter
\def\diam{\mathop{\operator@font diam}\nolimits}
\makeatother

\usepackage{pdflscape}

\newcommand{\mm}{\operatorname{m}}

\begin{document}
\title{Measurable tilings by abelian group actions}

	\author{Jan Greb\'ik}
	\address{University of Warwick Mathematics Institute, Coventry CV4 7AL, UK.}
	\email{jan.grebik@warwick.ac.uk}
	
	\author{Rachel Greenfeld}
	\address{Institute for Advanced Study, Princeton, NJ 08540.}
	\email{greenfeld.math@gmail.com}
	
	\author{V\'aclav Rozho\v{n}}
	\address{Department of Computer Science, ETH, Zurich, 8092, Switzerland}
	\email{rozhonv@ethz.ch}
	
	\author{Terence Tao}
	\address{UCLA Department of Mathematics, Los Angeles, CA 90095-1555.}
	\email{tao@math.ucla.edu}

\maketitle

\begin{abstract}
    Let $X$ be a measure space with a measure-preserving action $(g,x) \mapsto g \cdot x$ of an abelian group $G$.  We consider the problem of understanding the structure of measurable tilings $F \odot A = X$ of $X$ by a measurable tile $A \subset X$ translated by a finite set $F \subset G$ of shifts, thus the translates $f \cdot A$, $f \in F$ partition $X$ up to null sets.  Adapting arguments from previous literature, we establish a ``dilation lemma'' that asserts, roughly speaking, that $F \odot A = X$ implies $F^r \odot A = X$ for a large family of integer dilations $r$, and use this to establish a structure theorem for such tilings analogous to that established recently by the second and fourth authors.  As applications of this theorem, we completely classify those random tilings of finitely generated abelian groups that are ``factors of iid'', and show that measurable tilings of a torus $\T^d$ can always be continuously (in fact linearly) deformed into a tiling with rational shifts, with particularly strong results in the low-dimensional cases $d=1,2$ (in particular resolving a conjecture of Conley, the first author, and Pikhurko in the $d=1$ case).
\end{abstract}

\section{Introduction}

In this paper we establish a ``dilation lemma'' and ``structure theorem'' for abelian measurable tilings, and apply this to obtain new results on  factor-of-iid tilings, as well as measurable translational tilings of tori.

\subsection{Dilation lemmas and structure theorems for abelian tilings}

Let $G = (G,\cdot)$ be a (discrete) group.  By a  \emph{(translational) tiling} $F \odot A = G$ of $G$, we mean a pair consisting of a finite subset $F$ of $G$ and a subset $A$ of $G$ such that the translates $f \cdot A \coloneqq \{ f \cdot a: a \in A \}$ of $A$ by $F$ partition $G$.  If $G = (G,+)$ is written using additive group notation instead of multiplicative group notation, we write $F \oplus A = G$ instead of $F \odot A = G$ (and $f+A$ instead of $f \cdot A$).  For instance, we have
$$ \{0,1\}^2 \oplus (2\Z)^2 = \Z^2.$$
See for instance \cite{kol,KM} for surveys on the topic of translational tilings.

One can also consider translational tilings involving multiple pairs $F_i, A_i$.  For instance, if $G = (G,+)$ is an additive group, we write
$$ (F_1 \oplus A_1) \uplus \dots \uplus (F_k \oplus A_k) = G$$
for various finite subsets $F_1,\dots,F_k \subset G$ and $A_1,\dots,A_k \subset G$ if the translates $f_i + A_i$ for $i=1,\dots,k$ and $f_i \in F_i$ partition $G$.  Similarly, if $G$ is a multiplicative group.

In the case when the group $G$ is abelian and one is tiling by only one tile, there is a remarkable \emph{dilation phenomenon} \cite{tijdeman215decomposition,szegedy1998algorithms,iosevich2017fuglede,haessig2018tiling,kari2020algebraic,bhattacharya2020periodicity,GreenfeldTao} that asserts, roughly speaking, that the tiling $F \odot A = G$ implies the tiling $F^r \odot A = G$ for many integers $r$, where $F^r \coloneqq \{ f^r: f \in F \}$.  (Again, when the group is written additively, one would write $rF \oplus A = G$ in place of $F^r \odot A = G$.)  In \cite{GreenfeldTao} it was shown that upon averaging in $r$, this dilation invariance can be exploited to establish structural properties of such tilings.\footnote{A qualitatively similar conclusion regarding the spectral measure of a measure-preserving system associated to a tiling was obtained in \cite[Lemma 3.2]{bhattacharya2020periodicity}.}

\begin{theorem}[Structure of tilings of $\Z^d$]\label{structure-zd}  Let $d \geq 1$, and suppose that $F \oplus A = \Z^d$ for some finite set $F \subset \Z^d$ and some $A \subset \Z^d$.
\begin{itemize}
    \item[(i)] (Dilation lemma) One has $rF \oplus A = \Z^d$ whenever $r$ is a natural number coprime to all primes less than or equal to the cardinality $|F|$ of $F$.
    \item[(ii)] (Structure theorem) If we normalize $0 \in F$, then we have a decomposition
    $$ \one_A = \one_{\Z^d} - \sum_{f \in F \backslash \{0\}} \varphi_f$$
    where $\one_A$ denotes the indicator function of $A$ (thus $\one_A(x)=1$ when $x \in A$ and $\one_A(x)=0$ otherwise), and for each $f \in F \backslash \{0\}$, $\varphi_f \colon \Z^d \to [0,1]$ is a function which is $qf$-periodic (i.e., $\varphi_f(qf+ x)=\varphi_f(x)$ for every $x\in\Z^d$), where $q$ is the product of all the primes less than or equal to $2|F|$. 
\end{itemize}
\end{theorem}

\begin{proof}  Part (i) follows from \cite[Lemma~3.1(ii)]{GreenfeldTao}, while part (ii) follows from \cite[Theorem 1.7]{GreenfeldTao}.  The results in fact extend also to ``periodic level tilings''; see \cite{GreenfeldTao} for details.
\end{proof}

\cref{structure-zd} can then be used to obtain several new results about tilings of $\Z^d$ for low values of $d$, for instance establishing that all tilings of $\Z^2$ are weakly periodic; see \cite{GreenfeldTao} for details.

In this paper we extend the dilation and structure theorem to the context of \emph{measurable tilings}.  In this setting, we have a (discrete) group $\Gamma = (\Gamma,\cdot)$ acting on some other measure space $X = (X,{\mathcal X},\mu)$ in a measure-preserving action $\gamma \colon x \mapsto \gamma \cdot x$ for each $\gamma \in \Gamma$, thus
$$ \id_\Gamma \cdot x = x$$
for all $x \in X$ (where $\id_\Gamma$ denotes the group identity in $\Gamma$), and
$$ (\gamma \gamma') \cdot x = \gamma \cdot (\gamma' \cdot x)$$
for all $\gamma,\gamma' \in \Gamma$ and $x \in X$.  By a \emph{measurable tiling} $F \odot A =_{\ae} X$ of $X$, we mean a measurable subset $A$ of $X$ and a finite set $F$ of $\Gamma$ such that the dilates $f \cdot A \coloneqq \{ f \cdot a: a \in A \}$ of $A$ for $f \in F$ partition $X$ up to $\mu$-null sets.  Again, if the group $\Gamma$ is written additively, we write $F \oplus A$ instead of $F \odot A$.  For instance, if $\R^2$ acts on the torus $\T^2 = \R^2/\Z^2$ by translation, then we have
$$ \{0,1/2\}^2 \oplus ([0,1/2]^2 \hbox{ mod } \Z^2) =_\ae \T^2.$$

Our first result is the following analogue of \cref{structure-zd} in this setting:

\begin{theorem}[Structure of abelian measurable tilings]\label{structure-gen}  Let $\Gamma = (\Gamma,\cdot)$ be an abelian group acting on a measure space $X = (X,{\mathcal X},\mu)$ in a measure-preserving way, and suppose that $F \odot A =_\ae X$ for some finite set $F \subset \Gamma$ and some measurable $A \subset X$.
\begin{itemize}
    \item[(i)] (Dilation lemma) One has $F^r \odot A =_\ae X$ whenever $r$ is an integer number coprime to  $|F|$. 
    \item[(ii)] (Structure theorem) Suppose that $\mu(X)$ is finite.  Let $q=|F|$. Then we have a decomposition
    \begin{equation}\label{onex}
    \one_X =_\ae \sum_{f \in F} \varphi_f
    \end{equation}
    where we use $\varphi =_\ae \psi$ to denote the assertion that $\varphi,\psi$ agree $\mu$-almost everywhere, and for each $f \in F$, $\varphi_f \colon X \to [0,1]$ is a measurable function which is $f^q$-invariant up to null sets (thus $\varphi_f(f^q \cdot{-})) =_\ae \varphi_f$).  Furthermore, for all $f \in F$, we have $\int_X \varphi_f\ d\mu = \mu(A)$, and if $f^q \cdot A =_\ae A$ then $\varphi_f =_\ae \one_{f \cdot A}$.\footnote{Note that, as opposed to \cref{structure-zd}(ii), we no longer require $0\in F$ for the structure theorem \ref{structure-gen}(ii).}
\end{itemize}
\end{theorem}

We prove part (i) of this theorem in  \cref{sec:dilation}, and part (ii) of this theorem in \cref{sec:structure}, by adapting the arguments from \cite{GreenfeldTao}.  The main ingredient in the proof of \cref{structure-gen}(i) is the Frobenius identity $(a+b)^p = a^p + b^p$, valid in any commutative ring of characteristic $p$, and part (ii) will be derived from part (i) and the mean ergodic theorem.  

The requirement that $\mu(X)$ be finite in \cref{structure-gen}(i), as well as the requirement that the action of $G$ is measure-preserving, can be relaxed, as long as we also drop the conclusion that the $\varphi_f$ have mean $\mu(A)$; see \cref{app}.  In particular, we recover the result in \cref{structure-zd}(ii) this way despite the fact that $\Z^d$ has infinite counting measure.  We also remark that a result very similar to \cref{structure-gen}(i), though using somewhat different notation, was proven in \cite[Proposition 3.1]{bhattacharya2020periodicity}.
    
Informally, \cref{structure-gen}(ii) allows one to describe sets $A$ that tile a finite measure space $X$ in terms of auxiliary functions $\varphi_f$ that enjoy some ``one-dimensional'' invariance properties.  As such, this result will be particularly useful when the space $X$ also has very low dimension, and in particular when $X$ is the unit circle $\T$ or the two-torus $\T^2$, although it also gives some non-trivial results in higher dimension.

\subsection{First application: factor of iid tilings}

We now turn to applications of \cref{structure-gen}.  We first consider a class of random tilings of a group known as \emph{factor of iid tilings}.

\begin{definition}[Factor of iid]\label{def:factor}  Let $G = (G,+)$ be a group.  For each element $x$ of $G$, let $\lambda(x)$ be an iid element of the unit interval $[0,1]$, thus the $(\lambda(x))_{x \in G}$ are jointly independent random variables, each drawn uniformly at random from $[0,1]$.  A random subset $A$ of $G$ is said to be a \emph{factor of iid process} if there exists a Borel measurable function $\Phi \colon [0,1]^G \to \{0,1\}$ such that
\begin{equation}\label{process-eq} \one_A(x_0) = \Phi\left( (\lambda(x_0+x))_{x \in G} \right)
\end{equation}
almost surely for all $x_0 \in G$. (In particular, $A$ is a stationary process.)  More generally, a finite collection $A_1,\dots,A_k$ of random subsets of $G$ is a \emph{(joint) factor of iid process} if there exist Borel measurable functions $\Phi_1,\dots,\Phi_k \colon [0,1]^G \to \{0,1\}$ such that
$$
\one_{A_i}(x_0) = \Phi_i\left( (\lambda(x_0+x))_{x \in G} \right)
$$
almost surely for all $x_0 \in G$. 

If $F_1,\dots,F_k$ are finite subsets of $G$, a \emph{factor of iid tiling} of $G$ by $F_1,\dots,F_k$ is a joint factor of iid process $A_1,\dots,A_k$ such that
$$ (F_1 \oplus A_1) \uplus \dots \uplus (F_k \oplus A_k) = G$$
almost surely.
\end{definition}

Informally, a factor of iid tiling is a tiling which is generated in a ``local'' fashion, in the sense that the behavior of the tiling sets $A_1,\dots,A_k$ in some finite region $\Omega$ of $G$ is primarily determined by the random variables $\lambda(x)$ for $x$ near $\Omega$.  We illustrate the concept with the following example\footnote{See  \cite{KS,K}, where it was shown that tilings by two tiles can model any free ergodic $\Z$ action (upto a certain entropy threshold). See also \cite{Robinson-Sahin,KQS} for  results about tiling of orbits of any free, measure preserving $\Z^d$ (or $\R^d$) actions by a fixed number of tiles (depending on $d$).}:

\begin{example}\label{23-tile}  We can generate a factor of iid tiling of the integers $\Z$ by the tiles $F_1 \coloneqq \{0,1\}$, $F_2 \coloneqq \{0,1,2\}$ by performing the following procedure.
\begin{itemize}
    \item[(i)]  First we generate jointly independent random variables $\lambda(x) \in [0,1]$ for all $x \in \Z$.
    \item[(ii)]  We construct the factor of iid process
    $$ S \coloneqq \{ x \in \Z: \lambda(x) < \lambda(x-1),\lambda(x+1) \}$$
    of ``local minima'' of $\lambda$.  We enumerate $S=\{s_n\colon n\in \Z\}$ by order (i.e., $s_n<s_m$ if $n<m$). Note that $S$ is almost surely unbounded both above and below, and is $2$-separated, in the sense that $s_{n+1} - s_n \geq 2$ for any two consecutive elements $s_n, s_{n+1}$ of $S$.
    \item[(iii)]  Using the process $S=\{s_n\colon n\in \Z\}$, we construct the factor of iid process $S'$ by 
    $$S'\coloneqq \{s_n+2j\in \Z \colon n,j\in\Z, s_n\le s_n+2j\le s_{n+1}-2\}.$$
    We enumerate $S'=\{s'_n\colon n\in \Z\}$ by order.
      Note that $S'$ is also almost surely unbounded above and below, and one has $2 \leq s'_{n+1}-s'_n \leq 3$ for any consecutive elements $s'_n, s'_{n+1}$ of $S'$. 
    \item[(iv)]  Using the process $S'=\{s'_n\colon n\in \Z\}$, we construct the joint factor of iid process $A_1, A_2$ by setting $A_1$ to consist of those elements $s'_n$ of $S'$ for which $s'_{n+1}-s'_n = 2$, and $A_2$ to consist of those elements of $s'_n$ of $S'$ for which $s'_{n+1}-s'_n = 3$.
\end{itemize}
One then easily verifies that $A_1, A_2$ is a factor of iid tiling of $\Z$ by $F_1,F_2$.

On the other hand, there is no factor of iid tiling $F_1 \oplus A_1 = \Z$ of the integers $\Z$ just by $F_1 = \{0,1\}$, due to the ``rigid'' nature of this tiling equation.  Indeed, the only possible values of $A_1$ are the even integers $2\Z$ or the odd integers $2\Z+1$.  The events $0 \in A_1$, $1 \in A_1$ then complement each other and thus must each occur with probability $1/2$ by stationarity.  On the other hand, for any integer $N$, we have $0 \in A_1$ if and only if $2N \in A_1$; sending $N \to \infty$ we conclude that $0 \in A_1$ is a tail event, contradicting the Kolmogorov zero-one law.  A similar argument shows that there is no factor of iid tiling that involves only the tile $F_2 = \{0,1,2\}$.
\end{example}

The above argument can be strengthened to show that the tiling with the two tiles $F_1, F_2$ is possible not only as a factor of iid, but even in other, more restrictive, models. 
These include so-called finitary factors of iid, finitely dependent processes, local distributed algorithms etc. \cite{FinDepCol,HolroydSchrammWilson2017FinitaryColoring,brandt_grids,GJKS,grebik_rozhon2021toasts_and_tails}.
Similarly, a tiling with just $F_1$, or just $F_2$, is not possible in any of these models. 
Our first application of \cref{structure-gen}, which we prove in \cref{sec:factor}, shows that this latter phenomenon is quite general, in that in any finitely generated abelian group $G$ there are only very few tiles $F$ that admit a factor of iid tiling:

\begin{theorem}[Tiles admitting a factor of iid tiling]\label{thm:tile-factor}  Let $G = (G,+)$ be a finitely generated abelian group, thus without loss of generality we may take $G = \Z^d \times G_0$ for some natural number $d$ and finite abelian group $G_0 = (G_0,+)$.  Let $F$ be a finite subset of $G$.  Then the following are equivalent:
\begin{itemize}
    \item[(i)]  There exists a factor of iid tiling $F \oplus A = G$ of $G$ by $F$. 
    \item[(ii)]  $F$ is of the form $F = \{x_0\} \times F_0$ for some $x_0 \in \Z^d$ and $F_0 \subset G_0$, such that $G_0$ admits a tiling $F_0 \oplus A_0 = G_0$ by $F_0$.
\end{itemize}
\end{theorem}

Thus for instance the only tiles $F$ that admit a factor of iid tiling of $\Z^d$ are the singleton tiles $F = \{x\}$.

\begin{remark}
After the submission of the paper, Tim Austin suggested a simpler proof of a stronger version of \cref{thm:tile-factor} saying that \emph{(i)} and \emph{(ii)} in the theorem are equivalent to the third statement:

 \emph{(iii)} Let $A$ be the stationary point process  on $G$ such that $F\oplus A=G$. If $A$ is not trivial then $A$ has positive topological entropy.

The direction ``\emph{(ii)} implies \emph{(iii)}'' is similar to our proof of ``\emph{(ii)} implies \emph{(i)}'' in \cref{sec:factor}. The direction ``\emph{(iii)} implies \emph{(ii)}'' is an immediate corollary of \cref{structure-gen}, but can also be deduced by a more elementary argument similar to the proof of \cite[Lemma 2.15]{cyr-kra}.
\end{remark}

\subsection{Second application: measurable tilings of tori}

Our second application of \cref{structure-gen} concerns measurable tilings $F \oplus A =_\ae \T^d$ of a torus $\T^d \coloneqq \R^d/\Z^d$ using the standard translation action of $\R^d = (\R^d,+)$, thus $F$ is a finite subset of $\R^d$ and $A$ is a measurable subset of $\T^d$.  We say that such a tiling is \emph{rational} if the set $F-F=\{f'-f\colon  f,f'\in F\}$ lies in $\Q^d$, that is to say that all the shifts differences $f'-f$, $f',f \in F$ have rational coordinates.  Not all measurable tilings are rational; however, our main result below shows that all measurable tilings can be continuously deformed to a rational tiling, with the results particularly strong in the low dimensional cases $d=1,2$.  More precisely, we have

\begin{theorem}[Measurable tilings of a torus]\label{thm:torus}  Let $d \geq 1$, and suppose that we have a measurable tiling $F \oplus A =_\ae \T^d$ of the $d$-torus by some finite subset $F = \{f_1,\dots,f_n\}$ of $\R^d$ and some measurable subset $A$ of $\T^d$.  Then there exists a rational tiling $F^0 \oplus A =_\ae \T^d$ of the $d$-torus by some finite subset $F^0 = \{f_1^0,\dots,f_n^0\}$ of $\Q^d$, obeying the following additional properties:
\begin{itemize}
    \item [(i)]  If we define the velocities $v_i \coloneqq f_i - f_i^0$ for $i=1,\dots,n$ and the sets $F^t = \{ f_1^0 + tv_1, \dots, f_n^0 + tv_n \}$ for all $t \in \R$, then we have $F^t \oplus A =_\ae \T^d$ for all $t \in \R$.  In particular, one can continuously (and linearly) deform the original tile set $F = F^1$ to the rational tile set $F^0$ while retaining the measurable tiling property throughout.
    \item[(ii)]  If $d=2$ and we impose the normalization $0 \in F$, then all the velocities $v_i$ are scalar multiples $v_i = \alpha_i v$ of a single  vector $v \in \Z^2$ for some real numbers $\alpha_1,\dots,\alpha_n$.  Furthermore, we can partition $F$ into subsets $F_1,\dots,F_k$ such that for each $1 \leq j \leq k$, the elements of $F_j$ have the same velocity (thus $\alpha_i = \alpha_{i'}$ whenever $f_i,f_{i'} \in F_j$), and the set $F_j \oplus A \coloneqq \{ f_i + a: f_i \in F_j, a \in A \}$ is $\R v$-invariant in the sense that $tv + F_j \oplus A =_\ae F_j \oplus A$ for every $t \in \R$.
    \item[(iii)]  If the hypotheses are as in (ii), and furthermore the tile $A$ is open and connected, then we can furthermore assume that either all the velocities $v_i$ vanish (so in particular $F=F^0$ is rational), or else for each $1 \leq j \leq m$, the set $F_j \hbox{ mod } \Z^2$ lies in a coset of $\R v \hbox{ mod } \Z^2$, and all the $F_j$ have the same cardinality.
    \item[(iv)]  If $d=1$, then $F$ is rational; in other words, we have $F = F^0 + v$ for some $v \in \R$.
\end{itemize}
\end{theorem}

See \cref{fig:two_examples} for examples of tilings in cases (ii) and (iii). 
Informally, one can ``slide'' any measurable tiling of a torus by a single tile $A$ into a rational tiling by assigning each copy $f_i+A$ of the tile a constant velocity $v_i$ and propagating the tile backwards in time by one unit.  In two dimensions (with the normalization $0 \in F$) one can make the velocities parallel, and if the tile is additionally open and connected the tiling is either rational to begin with, or one can slide individual ``rows'' of the tiling separately.  Finally, we show that in one dimension the tiling is always rational. This gives a positive answer to a conjecture from \cite[Section 6]{conley_grebik_pikhurko2020divisibility_of_spheres}.  This conjecture can also be resolved by adapting arguments in \cite{leptin-muller,Lagarias-Wang}; see \cref{LW-remark}.

We prove parts (i), (ii), (iii), (iv) of \cref{thm:torus} in \cref{sec:highdim}, \cref{sec:twodim}, \cref{sec:connected}, and \cref{sec:onedim} respectively.  

We illustrate \cref{thm:torus} with some simple examples in dimensions one, two, and three:

\begin{example}[One dimension]
Let $d=1$ and $A \coloneqq [0,1/2] \hbox{ mod } \Z$.  Then a measurable tiling $F \oplus A =_\ae \T$ of $A$ necessarily takes the form $F = \{ v, m+1/2+v\}$ for some real number $v$ and integer $m$.  If we then take $F^0 \coloneqq \{ 0, m+1/2\}$, we see that $F = F^0 + v$ is a translate of the  set $F^0 \subset \Q$, thus rational, and that the other translates $F^t = F^0 + tv$ also give a measurable rational tiling: $F^t \oplus A =_\ae \T$.
\end{example}

\begin{figure}
    \centering
    \includegraphics[width = .9\textwidth]{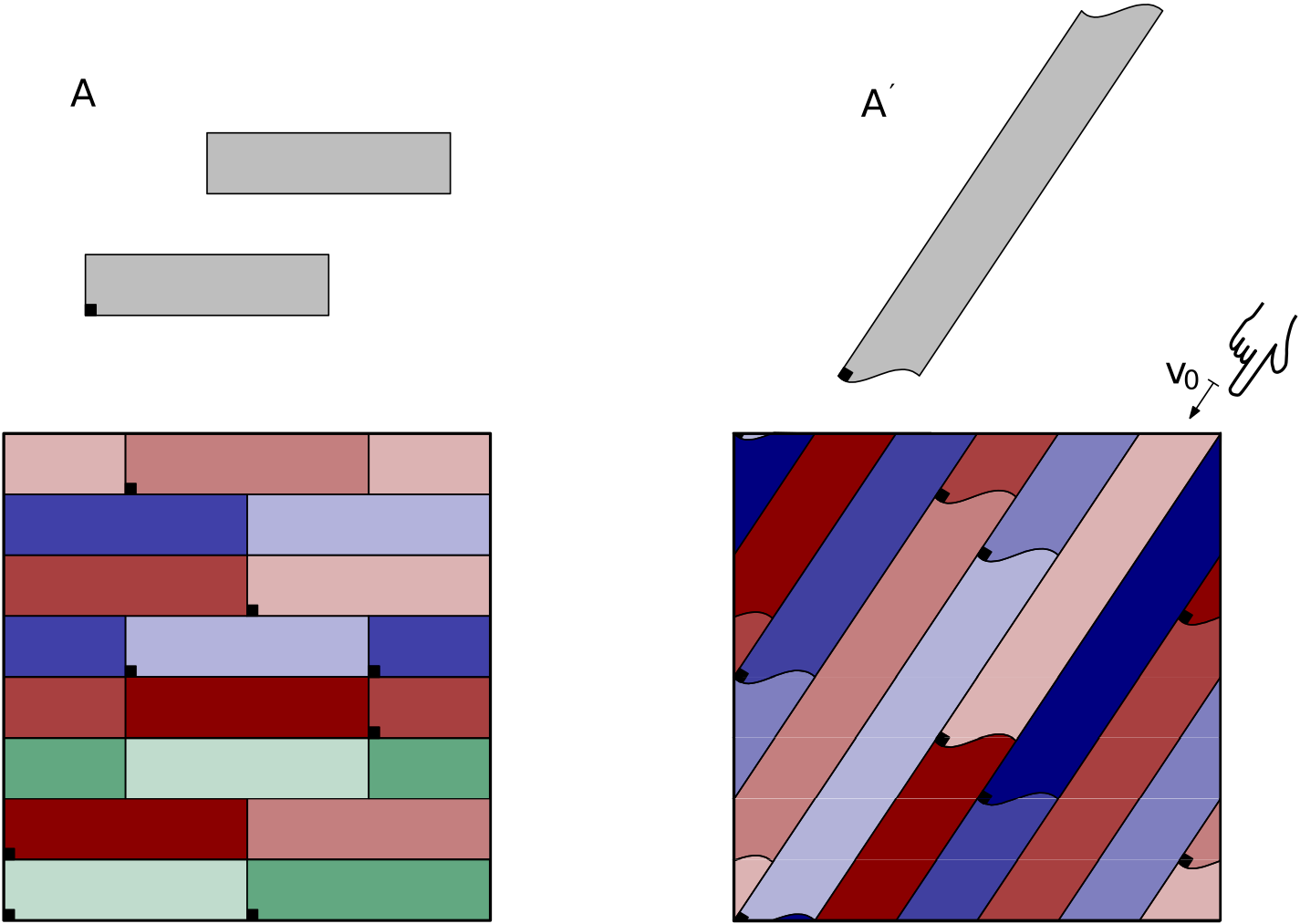}
    \caption{\\
    Left: A measurable tiling of $\T^2$ (depicted here using the fundamental domain $[0,1)^2$) by a disconnected tile $A$ by a set $F^0 = F^0_1 \cup F^0_2 \cup F^0_3$ defined in \cref{ex:disconnected} and denoted here by black squares. 
    The shifts $F^0_1, F^0_2, F^0_3$ generate green, blue, and red tiles.\\
    Right: A measurable tiling of $\T^2$ by an open and connected tile $A'$ by a set $F'^1 = F'^1_1 \cup F'^1_2$ denoted here by black squares. This time, the set $F'^1_2$ generating the red tiles is not necessarily a subset of $\Q^2$. However, sliding all red tiles in the direction of a vector $v_0$ (moving in the direction of the finger), we may enforce that the new coordinate set $F'^0_2$ is rational. }
    \label{fig:two_examples}
\end{figure}

\begin{example}[Two dimensions, disconnected] \label{ex:disconnected}
Let $d=2$ and $A$ be the set
$$ A \coloneqq \left( (0,1/2) \times (0,1/8) \;\cup\; (1/4,3/4) \times (1/4,3/8) \right) \hbox{ mod } \Z^2,$$
which is a disconnected open subset of $\T^2$; see the left half of \cref{fig:two_examples}.  If we define $F^0 \coloneqq F^0_1 \cup F^0_2 \cup F^0_3$ with
\begin{align*}
F^0_1 &\coloneqq \{ (0, 0), (1/2, 0) \} \\
F^0_2 &\coloneqq \{ (1/4, 1/2), (3/4,1/2) \} \\
F^0_3 &\coloneqq \{ (0,1/8), (3/4,3/8), (1/2,5/8), (1/4,7/8)\}
\end{align*}
and set $v_0 \coloneqq (1,0)$, then we have a measurable tiling $F^0 \oplus A =_\ae \T^2$, with each $F^0_i \oplus A$ being $\R v_0$-invariant; see the left half of \cref{fig:two_examples}.  If we then let $\alpha_1,\alpha_2,\alpha_3$ be arbitrary real numbers and set $F \coloneqq F_1 \cup F_2 \cup F_3$ where
$$ F_i \coloneqq F^0_i + \alpha_i v_0$$
for $i=1,2,3$, we see that we also have a measurable tiling $F \oplus A =_\ae \T^2$, which was obtained from $F^0$ by giving the tiles in $F^0_i$ a velocity of $\alpha_i v_0$ and then moving the tiles for a unit amount of time. 
Note that the set $F_3 \hbox{ mod } \Z^2$ is not contained in a single coset of $\R v_0 \hbox{ mod } \R^2$.
\end{example}

\begin{example}[Two dimensions, connected]  Let $d=2$ and $A \coloneqq (0,1/2)^2 \hbox{ mod } \Z^2$; this is an open connected set.  Let $\alpha$ be an irrational number.  Then the set
$$ F \coloneqq \{ (0,0), (0+1/2,0), (\alpha,1/2), (\alpha+1/2,1/2)\}$$
generates a measurable tiling $F \oplus A =_\ae \T^2$ of the torus $\T^2$.  If for every real $t$ we set
$$ F^t = \{ (0,0), (1/2,0), (t,1/2), (t+1/2,1/2)\}$$
then $F^t \oplus A = _\ae \T^2$ is a measurable tiling for every real number $t$, which is rational when $t=0$.  Also, if we set $v_0 = (1,0)$, and partition $F = F^1$ into
$$ F_1 = \{ (0,0), (1/2,0)\}; \quad F_2 = \{(\alpha,1/2), (\alpha+1/2,1/2)\}\},$$
then we can give the elements of $F_1$ a zero velocity, and the elements of $F_2$ a velocity of $\alpha v_0$, and the sets 
$$ F_1 \oplus A =_\ae \T \times (0,1/2) ; \quad
F_2 \oplus A =_\ae  \T \times (1/2,1)$$
are $\R v_0$-invariant; informally, this means that one can independently ``slide'' the sets $F_1, F_2$ along the direction $v_0$ without destroying the measurable tiling property. Note that $F_1 \hbox{ mod } \Z^2$ and $F_2 \hbox{ mod } \Z^2$ both lie on cosets of $\R v_0 \hbox{ mod } \Z^2$.

A more complicated example of a connected tile in two dimensions is depicted on the right-hand side in \cref{fig:two_examples}. 
\end{example}

\begin{example}[Three dimensions]\label{three-dim}  Let $d=3$ and $A \coloneqq [0,1/2]^3 \hbox{ mod } \Z^3$.  Let $\alpha,\beta,\gamma$ be irrational numbers.  For every real number $t$, set
\begin{align*}
     F^t \coloneqq \{ 
     &(0, t\alpha, 0), (0, t\alpha+1/2, 0),\\
     &(1/2, 0, t\beta), (1/2, 0, t\beta+1/2),\\
     &(t\gamma, 1/2, 1/2), (t\gamma+1/2, 1/2, 1/2),\\
     &(0, 0, 1/2), (1/2, 1/2, 0)\}.
\end{align*}One can then verify that $F^t \oplus A =_\ae \T^3$ for all real $t$ (see, \cref{fig:cube}).  In particular, one can ``slide'' the irrational tiling $F = F^1$ into the rational tiling $F^0$ without destroying the tiling property, with the elements of $F$ being given velocities proportional to $(1,0,0), (0, 1, 0)$ and $(0,0,1)$, respectively. 
\end{example}

\begin{figure}
    \centering
    \includegraphics[width = .4\textwidth]{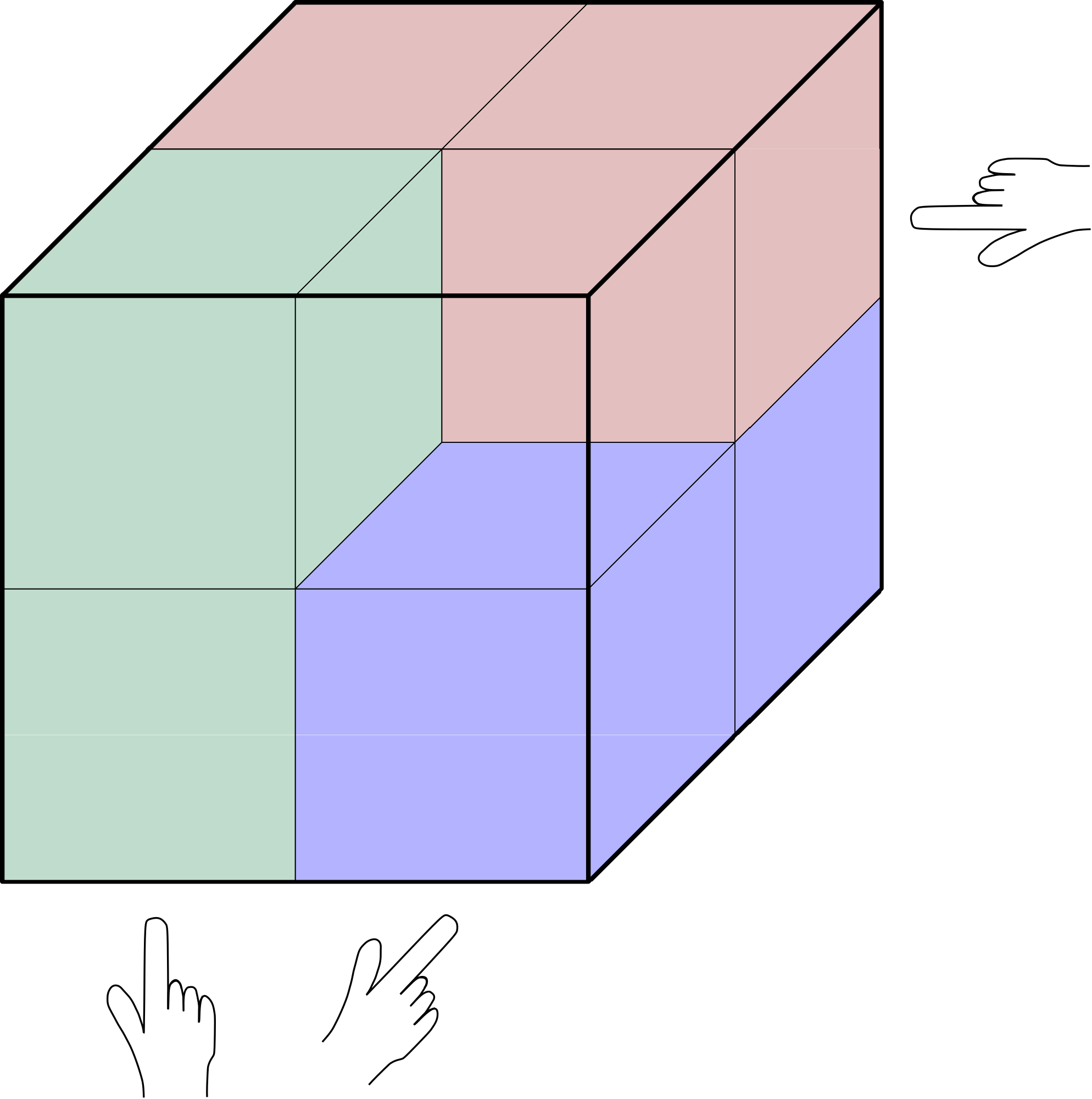}
    \caption{A measurable tiling of $\T^3$ by eight cubes $[0,1/2)^3$ from \cref{three-dim}. Only six cubes are colored; the green, blue, and red boxes are shifted in the direction $(1, 0, 0), (0, 1, 0), $ and $(0, 0, 1)$. }
    \label{fig:cube}
\end{figure}

\section{A measurable dilation lemma}\label{sec:dilation}

In this section we establish \cref{structure-gen}(i).  Our arguments here will be a modification\footnote{In the model case $\Gamma=\Z^d$, one can in fact derive \cref{structure-gen}(i) directly from \cite[Lemma~3.1]{GreenfeldTao} by applying that lemma to the sets $\{ \gamma \in \Z^d: \gamma \cdot x \in A \}$, which form a tiling of $\Z^d$ by $F$ for almost every $x \in X$; we leave the details of this argument to the interested reader.} of those used to establish \cite[Lemma~3.1]{GreenfeldTao}.

It will be convenient to introduce the language of convolutions.
Let $L^0(X)$ denote the space of measurable functions $f \colon X \to \R$, up to almost everywhere equivalence, and let $\R \Gamma$ denote the group ring of $\Gamma$ over $\R$, which we write as the space of finitely supported functions $w \colon \Gamma \to \R$ from $\Gamma$ to the reals.  
With this representation, the multiplication operation on $\R \Gamma$ becomes the usual convolution operation $*$:
$$ w_1 * w_2(\gamma) \coloneqq \sum_{\gamma' \in \Gamma} w_1(\gamma') w_2( (\gamma')^{-1} \gamma);$$
note that only finitely many of the summands are non-zero.  This operation is bilinear and associative, and it is commutative whenever $\Gamma$ is abelian. 
We can also define the convolution $w * f$ of an element $w \in \R \Gamma$ of the group ring and a function $f \colon X \to \R$ by the formula
$$ w*f(x) \coloneqq \sum_{\gamma \in \Gamma} w(\gamma) f(\gamma^{-1} x);$$
again, only finitely many summands are non-zero, and from the invariance of $\mu$ we see that if $f$ is only given up to $\mu$-almost everywhere equivalence then $w*f$ is also well-defined up to $\mu$-almost everywhere equivalence.  Thus the convolution $w*f \in L^0(X)$ is also well-defined for $w \in \R \Gamma$ and $f \in L^0(X)$.  The ring $\R\Gamma$ can easily be seen to act on $L^0(X)$; in particular, we have
$$ (w_1*w_2)*f = w_1 * (w_2*f)$$
for all $w_1,w_2 \in \R\Gamma$ and $f \in L^0(X)$.

Note that if $F$ is a finite subset of $\Gamma$ and $A$ is a measurable subset of $X$, then $\one_F$ can be viewed as an element of $\R \Gamma$ and $\one_A$ can be viewed as an element of $L^0(X)$.  The tiling condition $F \odot A =_\ae X$ is then equivalent to the convolution identity
\begin{equation}\label{fax}
 \one_F * \one_A =_\ae \one_X
\end{equation}
holding in $L^0(X)$. 

We begin with the proof of \cref{structure-gen}(i) for $r > 0$.  We may assume that $A$ has positive measure, as the claim is trivial otherwise.  By induction and the fundamental theorem of arithmetic, it suffices to verify this claim in the case that $r$ is a prime $p$ with $p > |F|=n$, so long as we verify that $F^p$ has the same cardinality as $F$ (i.e., there are no collisions $f_1^p = f_2^p$ for distinct $f_1,f_2 \in F$).

We convolve both sides of \eqref{fax} by $p-1$ additional copies of $\one_F$, noting that $\one_F * \one_X = n \one_X$, to conclude that
$$ \one_F^{*p} * \one_A =_\ae n^{p-1} \one_X$$
in $L^0(X)$, where $\one_F^{*p}$ denotes the convolution of $p$ copies of $\one_F$.  The left-hand side is integer-valued, thus we may reduce both sides modulo $p$ and conclude from Fermat's little theorem that
\begin{equation}\label{fa1p}
\one_F^{*p} * \one_A =_\ae \one_X \hbox{ mod } p.
\end{equation}

The group algebra ${\mathbb F}_p \Gamma$ of functions $w \colon \Gamma \to {\mathbb F}_p$ is a commutative ring of characteristic $p$, and thus one has the Frobenius identity $(w_1+w_2)^{*p} = w_1^{*p} + w_2^{*p}$ in this ring for all $w_1,w_2 \in {\mathbb F}_p \Gamma$.  Writing $\one_F = \sum_{i=1}^n \delta_{f_i}$ as the sum of Kronecker delta functions, we conclude that
$$ \one_F^{*p} = \sum_{i=1}^n \delta_{f_i}^{*p} = \one_{F^p}$$
in ${\mathbb F}_p \Gamma$, where we temporarily view $F^p \coloneqq \{ f^p: f \in F\}$ as a multiset rather than a set, so that the indicator function $\one_{F^p}$ could theoretically take on values greater than one (although we shall shortly eliminate this possibility).  In other words,
$$ \one_F^{*p} = \one_{F^p} \hbox{ mod } p.$$
Since $\one_A$ is also integer-valued, we conclude that
$$ \one_F^{*p} * \one_A =_\ae \one_{F^p} * \one_A \hbox{ mod } p$$
pointwise everywhere in $X$.  Combining this with \eqref{fa1p}, we conclude that
$$ \one_{F^p} * \one_A =_\ae \one_X \hbox{ mod } p.$$
This implies that $$ \one_{F^p} * \one_A \geq _\ae \one_X.$$ Observe that
$$|F|=\one_{F} * \one_{F^p} * \one_A \geq _\ae \one_{F} * \one_X=|F|.$$ 
We conclude that
$$ \one_{F^p} * \one_A =_\ae \one_X.$$
Since $\one_X$ is bounded by $1$ and $A$ has positive measure, it is thus not possible for $\one_{F^p}$ to attain any value larger than one, and hence there are no collisions $f_1^p = f_2^p$ for distinct $f_1, f_2 \in F$.  We thus have the measurable tiling $F^p \odot A =_\ae X$, as claimed.

To conclude the proof of \cref{structure-gen}(i) for all suitable $r \in \Z$, it suffices by the first part to treat the case $r=-1$, that is to say that the translates $f_i^{-1} \cdot A, 1\leq i\leq n$ partition $X$ up to null sets. To show this, one can  adopt the arguments in \cite[Theorem 13]{szegedy1998algorithms}, \cite[Lemma 3.1]{kol} and \cite[Lemma 3.2]{GL}.
By hypothesis, we see that for any  $1\leq i<i'\leq n$,  the translates $f_i \cdot A, f_{i'} \cdot A$ are disjoint up to $\mu$-null sets; translating this by $(f_i f_{i'})^{-1}$ and using the abelian nature of $\Gamma$ we conclude that $f_{i'}^{-1} \cdot A, f_i^{-1} \cdot A$ are also disjoint up to $\mu$-null sets.  Thus   $\one_{F^{-1}}*\one_A\le_\ae \one_X$. On the other hand we have $\one_F*\one_A=_\ae\one_X$, and $|F|=|F^{-1}|=n$. Thus, using again the abelian nature of $\Gamma$, we have $$\one_F*(\one_X-\one_{F^{-1}}*\one_A)=_\ae n\one_X -\one_{F^{-1}}*(\one_F*\one_A)=_\ae 0.$$
Thus, as both $\one_F$ and $(\one_X - \one_{F^{-1}} * \one_A)$ are non-negative and $|F|>0$, we must have $$\one_X-\one_{F^{-1}}*\one_A =_\ae 0,$$
 and the claim follows.

\begin{remark}\label{non-ab} The dilation lemma fails when the group $\Gamma$ is non-abelian.  For instance, consider the group $\Gamma = \Z \times G$ for some (non-abelian) finite group $G = (G,\cdot)$ (and using the additive group law on $\Z$), acting on $X = \Gamma$ (equipped with counting measure) by left translation; one can also take $X$ to be a quotient $\Z/N\Z \times G$ of $\Gamma = \Z \times G$ if desired to ensure that $X$ has finite measure.  Let $Ha$ be some right coset of a proper subgroup $H$ of $G$, and consider the finite set $F \subset \Gamma$ defined by
$$ F \coloneqq (\{0\} \times Ha) \cup (\{1\} \times (G \backslash H)).$$
Observe that if $A \subset X$ is a set of the form
\begin{equation}\label{A-def}
A \coloneqq \{ (n,g_n): n\in \Z\}
\end{equation}
for some sequence $g_n$ of elements of $G$, then the translates $f \cdot A, f \in F$ partition $X$ if and only if one has the constraint
\begin{equation}\label{gnn}
 g_{n-1} \in H a g_n
 \end{equation}
for all $n \in \Z$.  If we take $X$ to be $\Z/N\Z \times G$ instead of $\Z \times G$, the above discussion still applies, but with $n$ now ranging in $\Z/N\Z$ rather than $\Z$.
In the non-abelian setting, one can easily construct examples\footnote{For instance, one can take $G=S_3$, $H$ to be a subgroup of $S_3$ of order two, and $a$ to be an element not in $H$.} in which $HaHaHa=G$, in which case the constraint \eqref{gnn} gives no relationship whatsoever between $g_{n+r}$ and $g_n$ for $r \ge 3$.  In particular, for such $r$ there is no dilated tile of the form
$$ F_r = (\{0\} \times E) \cup (\{r\} \times E')$$
for some non-empty $E, E' \subset G$ with the property that $F \odot A = X$ implies the $F_r \odot A = X$.  A similar analysis shows that the assertions $F \odot A = X$ and $F^{-1} \odot A = X$ are inequivalent.  This example indicates that no reasonable analogue of the dilation lemma holds in this setting.  This example also shows that non-abelian tiling problems with one tile can be ``local'' in various senses; see, for instance,  \cref{fid-counter} below for a more precise statement.
\end{remark}

\begin{remark}\label{grf}  As showed in  \cite[Lemma~3.1]{GreenfeldTao}, one can generalize the dilation lemma by requiring the tiling to be a \emph{periodic level  tiling} rather than a partition up to null sets, by which we mean that  for every $1\le k\le |F|$, the level set $\{x\colon \one_F*\one_A(x)=k\}$ is \emph{periodic} in $X$ up to $\mu$-null sets, (where here a periodic  set is a set which is $\mu$-almost everywhere invariant with respect to an action of some  lattice).  The conclusion is then that there is a number $q$ (depending on $F$ and $\one_F*\one_A$) such that if  $r=1\mod q$ then $\one_{F^r}*\one_A=_\ae \one_{F}*\one_A$,  but now we permit collisions $f_1^r = f_2^r$ to occur.  We leave the details of this generalization to the interested reader.  
\end{remark}

\begin{remark}\label{quasi-rem}   \cref{structure-gen}(i) easily extends  to the setting in which the action of $\Gamma$ is quasi-invariant rather than invariant, which means that it maps $\mu$-null sets to $\mu$-null sets.  This can be accomplished simply by replacing $\mu$ with the (non-$\sigma$-finite) measure $\tilde \mu$ defined by setting $\tilde \mu(E)$ to equal $+\infty$ when $\mu(E)>0$ and to equal zero otherwise.
\end{remark}

\section{A measurable structure theorem}\label{sec:structure}

In this section, using \cref{structure-gen}(i),  we will establish  \cref{structure-gen}(ii).  Let the hypotheses be as in that theorem.  Applying part (i) of that theorem, we see that for any integer $r$ coprime to $q$, we have
$$ F^r \odot A =_\ae X,$$
which we rewrite as the assertion that
$$\one_X =_\ae  \sum_{f \in F} \one_{f^r \cdot A}.$$
Setting $r = 1+nq$ for $n=1,\dots,N$ and averaging, we conclude in particular that
$$\one_X =_\ae    \sum_{f \in F} \frac{1}{N} \sum_{n=1}^N \one_{(f^{q})^n \cdot f \cdot A}$$
for all $N$.  By the mean ergodic theorem, for each $f \in F
$, the averages
$\frac{1}{N} \sum_{n=1}^N \one_{(f^q)^n \cdot f \cdot A}$ converge in $L^1(X)$ to a $f^q$-invariant function $\varphi_f$; since these averages all have total mass $\mu(A)$, $\varphi_f$ does also.  It is also clear from construction that if $f^q \cdot A =_\ae A$ then $\varphi_f =_\ae \one_{f\cdot A}$.  The claim follows.

\begin{remark}
As it turns out, one can replace the requirement that the measure $\mu$ be finite to merely $\sigma$-finite, and also assume that the action is only quasi-invariant rather than measure-preserving, as long as we also drop the conclusion that the $\varphi_f$ have mean $\mu(A)$; see  \cref{app}.
\end{remark}

\begin{remark}\label{grf2} Using \cref{grf}, one can also extend the above structure theorem to periodic level tilings, and in particular, to \emph{level $k$ tilings}\footnote{Higher level tilings are studied in several places in the literature; see for instance \cite{kol}, \cite{gkrs}, \cite{swee-hong}.}  (by which we mean that almost every element of $X$ lies in precisely $k$ of the translates $f\cdot A$ for some $k\le |F|$), but now replacing $\one_X$ in \eqref{onex} with $k \one_X$.  However this identity \eqref{onex} is significantly less useful in the $k>1$ case due to the gap in values between $k \one_X$ and $\one_A$, which leaves more room for the functions $\varphi_f, f\in F$ to vary (cf., \cite[Theorem~1.3(ii)]{GreenfeldTao}).
\end{remark}

\section{Factor of iid tilings}\label{sec:factor}

In this section we prove \cref{thm:tile-factor}.

We first show that (ii) implies (i).  By translating $F$ we may assume without loss of generality that $x_0=0$.  By the hypothesis (ii), there exists a subset $A_0$ of $G_0$ such that the translates $f_0+A_0$, $f_0 \in F_0$ partition $G_0$.  Next, let $\lambda(\underline{x},g_0)\in[0,1], (\underline{x},g_0) \in \Z^d \times G_0$, be the iid random variables from \cref{def:factor}, and for each $\underline{x} \in \Z^d$, let $g_0(\underline{x})$ denote the element of $G_0$ which minimizes the quantity $\lambda(\underline{x},g_0(\underline{x}))$.  Clearly $g_0$ is almost surely well-defined as a function from $\Z^d$ to $G_0$.  We then form the random set
\begin{equation}\label{eq:A}
    A \coloneqq \{ (\underline{x}, g_0(\underline{x})+a_0): \underline{x} \in \Z^d; a_0 \in A_0 \}.
\end{equation} 
It is a routine matter to verify that $F \oplus A =_\ae G$ is a factor of iid tiling.  This proves (i).

Conversely, suppose that (i) holds.  Applying a translation, we may assume without loss of generality that $F$ contains the identity $(0,0)$ of $\Z^d \times G_0$.  Let $F \oplus A = G$ be a factor of iid tiling. Let $\Phi \colon [0,1]^G \to \{0,1\}$ be the measurable function obeying \eqref{process-eq}.  Observe that $[0,1]^G$ is a probability space with product measure $dm$ and a measure-preserving action of $G$ given by the translation action
$$ x_0 \cdot (\lambda_x)_{x \in G} \coloneqq
(\lambda_{x_0+x})_{x \in G}.$$
If we define the set $\tilde A \subset [0,1]^G$ by
$$ \tilde A \coloneqq \Phi^{-1}(\{1\})$$
then from \eqref{process-eq} one easily verifies that we have the tiling
$$ F \odot \tilde A =_\ae [0,1]^G.$$
Applying \cref{structure-gen}(ii), we obtain a decomposition
\begin{equation}\label{ga} \one_{[0,1]^G} =_\ae \one_{\tilde A} + \sum_{f \in F \backslash \{(0,0)\}} \varphi_f
\end{equation}
for some non-negative $qf$-invariant functions $\varphi_f \colon [0,1]^G \to [0,1]$ of mean $m(\tilde A)$; in particular, on integrating we have
$$ 1 = |F| m(\tilde A).$$
Suppose that there exists an element $f_*$ of $F$ that is not contained in $\{0\} \times G_0$.  Then the action of $qf_*$ on $[0,1]^G$ is ergodic (this follows for instance from the Kolmogorov zero-one law, since any $qf_*$-invariant subset in $[0,1]^G$ is measurable with respect to the tail algebra of $[0,1]^G$), and hence $\varphi_{f_*}$ is almost everywhere constant; since it has mean $m(\tilde A) = 1/|F|$, we thus have $\varphi_{f_*}=_\ae 1/|F|$.  From \eqref{ga} we thus have the inequality
\begin{equation}\label{ineq}
\one_{[0,1]^G} \geq \one_{\tilde A} + \frac{1}{|F|}
\end{equation}
almost everywhere, which is absurd since $\tilde A$ has positive measure.  Thus all elements of $F$ lie in $\{0\} \times G_0$, and so we may write $F = \{0\} \times F_0$ for some $F_0 \subset G_0$.

By hypothesis, there is a tiling $A$ of $\Z^d \times G_0$ by $\{0\} \times F_0$.  This implies that the set $A_0 \coloneqq \{ a_0: (0,a_0) \in A \}$ is a tiling of $G_0$ by $F_0$, giving (ii).
This completes the proof of \cref{thm:tile-factor}.

\begin{remark}
In the case when tiling a finitely generated abelian group with a tile which does not contain any non-trivial element of finite order (e.g.,  tiling $\Z^d$ with a non-trivial tile),   Theorem \ref{structure-gen} implies a stronger conclusion saying that the spectral measure of the tiling is supported on a finite union of subtorii; in particular, the tiling in this case is not weak-mixing in some directions. We thank the referee for this observation.
\end{remark}

\subsection{Some counterexamples}\label{fid-counter}
\subsubsection{Tiling by multiple tiles and tilings in non-abelian groups} In \cref{23-tile}, an example was given showing that \cref{thm:tile-factor} breaks down once two or more tiles are present.  
 We now give a modification of this example that shows that \cref{thm:tile-factor} also breaks down when the group $G$ is non-abelian.

Indeed, let $H$, $a$, $\Z \times G$, and $F$ be as in \cref{non-ab}, with $HaHaHa=G$.  We arbitrarily place  total ordering $<$ on $G$.  Despite the fact that the tile $F$ is not contained in a single fiber $\{x\} \times G$ of $\Z \times G$, one can construct a factor of iid tiling of $\Z \times G$ by $F$ by the following modification of the construction in \cref{23-tile}.

\begin{itemize}
    \item[(i)]  First we generate jointly independent random variables $\lambda(\underline{x},g) \in [0,1]$ for all $\underline{x} \in \Z$ and $g \in G$.  Then set $\underline{\lambda}(\underline{x}) \coloneqq \min_{g \in G} \lambda(\underline{x},g)$.
    \item[(ii)]  Similarly as in \cref{23-tile}, we construct the random set
    $$ S \coloneqq \{ \underline{x} \in \Z: \underline{\lambda}(\underline{x}) < \underline{\lambda}(\underline{x}-2),\underline{\lambda}(\underline{x}-1),\underline{\lambda}(\underline{x}+1),\underline{\lambda}(\underline{x}+2) \} \subset \Z.$$
    We enumerate $S=\{s_n\colon n\in \Z\}$ by order (i.e., $s_n<s_m$ if $n<m$).
    Observe that this set is almost surely unbounded both above and below, and that $s_{n+1} - s_n \geq 3$ for any two consecutive elements of $S$.
    \item[(iv)]  For each $s_n \in S$, we define $g_{s_n} \in G$ to be the element of $G$ that maximizes $\lambda(s_n, g)$, $g\in G$; this is almost surely well-defined.
    \item[(v)]  If $s_n, s_{n+1}$ are consecutive elements of $S$, we define $g_{s_n+j} = a^{-j} g_{s_n}$ for $1 \leq j \leq s_{n+1}-s_n-3$, and then define $g_{s_{n+1}-1} = b_0 g_{s_{n+1}-2}=b_0 b_1 g_{s_{n+1}-3}$, where $b_0, b_1$ is the lexicographically minimal pair of elements of $a^{-1} H$ such that
    $$ b_0b_1 g_{s_{n+1}-3} \in Ha g_{s_{n+1}}$$
    (such a pair exists since $HaHaHa = G$). Note from construction that for all $x \in \Z$, $g_{x}$ is now almost surely well-defined and obeys \eqref{gnn}. 
    \item[(vi)]  Finally, we let $A \subset \Z \times G$ be the set defined by \eqref{A-def}.
\end{itemize}

It is then a routine matter to verify that $F \odot A = \Z \times G$ is a (non-abelian) factor of iid tiling of $\Z \times G$ by $F$.  This shows that \cref{thm:tile-factor} breaks down once the group is non-abelian.

\subsubsection{Higher level tilings} Observe that \cref{thm:tile-factor} also breaks down once one considers tilings of level higher than one; in this setting there are (as noted in \cref{grf,grf2}) analogues of the dilation lemma and structure theorem, but the analogue of the inequality \eqref{ineq} no longer generates a contradiction.  Indeed, since $\one_F * \one_G = |F| \one_G$, every finite tile $F$ trivially has a factor of iid tiling of level $|F|$. In the latter example the tiling has entropy zero.

When $G=\Z^d$, any $k$-level factor of iid tiling has entropy zero\footnote{We thank the referee for this observation.}. Indeed, suppose that $A$ is a level $k$ factor of iid tiling of $\Z^d$ by $F$. A higher level version of \cref{structure-gen} (see \cref{grf2}) will give the generalization of \eqref{ga}:
\begin{equation}\label{ga2}  k\one_{[0,1]^{\Z^d}} =_\ae \one_{\tilde A} + \sum_{f \in F \backslash \{0\}} \varphi_f
\end{equation} 
where for  every $f \in F \backslash \{0\}$, $\varphi_f\colon \Z^d\to [0,k]$ is measurable $|F|f$-invariant and has mean $m(\tilde A) = k/|F|$. On the other hand, if  $f\neq 0\in F$, then, by the Kolmogorov zero-one law,  $\varphi_{f}$ is almost everywhere constant, thus $\varphi_{f}=_\ae k/|F|$.   
From \eqref{ga2} we thus have 
\begin{equation}
k\one_{[0,1]^{\Z^d}} =_\ae \one_{\tilde A} + \frac{k(|F|-1)}{|F|}
\end{equation}
which implies $k=|F|$ and $A=\Z^d$.

However, when $G=\Z^d\times G_0$ and $G_0$ is not trivial, there are non-vertical sets that admit \emph{non-trivial} factor of iid tilings of level higher than one; for instance, let $d\geq 1$, $k>1$,  $S$ be a subset of $\Z^d$ of cardinality $k$, and $G_0$ be some finite abelian group, then the set $S\times G_0$ admits a non-trivial (positive entropy) factor of iid tiling of level $k$ of $\Z^d\times G_0$ (to show this, one can adapt our construction \eqref{eq:A}, with $A_0=G_0$).

\section{Measurable tilings of a torus}

We now prove \cref{thm:torus}.  We begin with some easy consequences of \cref{structure-gen}:

\begin{lemma}[Initial properties]\label{init}  Let $F \oplus A =_\ae \T^d$ be a measurable tiling of a torus $\T^d$ by a finite set $F \subset \R^d$ and a measurable set $A \subset \T^d$. We normalize $0 \in F$.
\begin{itemize}
    \item[(i)]  (Weak rationality) For every $f \in F$, there exists $k \in \Z^d \backslash \{0\}$ such that $k \cdot f \in \Z$.
    \item[(ii)]  (Weak structure)  Up to sets of measure zero, one can write
    $$ \bigcup_{f \in F \cap \Q^d} (f + A) = \bigcap_{f \in F \backslash \Q^d} A_f$$
    where for each $f \in F \backslash \Q^d$, $A_f$ is a $qf$-invariant measurable subset of $\T^d$ for some natural number $q$ with $0 < \mu(A_f) < 1$.
\end{itemize}
\end{lemma}

\begin{proof}  We begin with (i).  From \cref{structure-gen}(ii) we have a decomposition
$$ \one_{\T^d} = \one_A + \sum_{f \in F \backslash \{0\}} \varphi_f$$
where for each $f \in F \backslash \{0\}$,  $\varphi_f \colon \T^d \to [0,1]$ is a measurable function of mean $\mu(A) = 1/|F|$ which is $qf$-invariant for some natural number $q$.  Now if $f \in F \backslash \{0\}$ is such that $k \cdot f \not \in \Z$ for all $k \in \Z^d \backslash \{0\}$, then by the Weyl equidistribution theorem, the action of $qf$ is ergodic, thus $\varphi_f$ is almost everywhere equal to a constant, which must be $1/|F|$.  Thus in this case we have the inequality
$$ \one_{\T^d} \geq \one_A + \frac{1}{|F|}$$
almost everywhere, which is a contradiction since $A$ has positive measure.  This proves (i).

Now we prove (ii).  Let $\tilde q$ be a natural number divisible by all primes less than or equal to $|F|$ such that $\tilde qf = 0 \hbox{ mod } \Z^d$ for all $f \in F \cap \Q^d$.  From \cref{structure-gen}(ii) we have
\begin{equation}\label{con}
 \one_{\T^d} = \sum_{f \in F \cap \Q^d} \one_{f + A} + \sum_{f \in F \backslash \Q^d} \tilde \varphi_f
 \end{equation}
 where for each $f \in F \backslash \Q^d$,  $\tilde \varphi_f \colon \T^d \to [0,1]$ is a measurable function which is $\tilde qf$-invariant and has positive mean.  If we let $A_f$ denote the complement of the support of $\tilde \varphi_f$, then $A_f$ is also $\tilde qf$-invariant and has measure less than $1$.  Note that, up to sets of measure zero, $\sum_{f \in F \backslash \Q^d} \tilde \varphi_f$ vanishes precisely on $\bigcap_{f \in F \backslash \Q^d} A_f$, and $\sum_{f \in F \cap \Q^d} \one_{f + A}$ is the indicator function of $\bigcup_{f \in F \cap \Q^d} f+A$.  The claim (ii) then follows from \eqref{con} (note that none of the $A_f$ can have zero measure since $\bigcup_{f \in F \cap \Q^d} (f + A)$ has positive measure).
 \end{proof}
 
\subsection{The one-dimensional case}\label{sec:onedim}

We can now easily establish the one dimensional case (iv) of \cref{thm:torus}.   Indeed, by translating $F \subset \R$ by a constant we may assume without loss of generality that $0 \in F$.  From \cref{init}(i) we then see that every element of $F \backslash \{0\}$ is rational, and \cref{thm:torus}(iv) follows.
 
\begin{remark}\label{LW-remark}  An alternate way to prove \cref{thm:torus}(iv) is as follows.  In  \cite[Theorem~2]{Lagarias-Wang}  it was proved that if $A\subset \R$ is bounded, Lebesgue measurable and has a zero  measure boundary and if $A\oplus R=_\ae\R$ for some $R\subset \R$, then the set $R'=\mu(A)^{-1}R$ must be rational, i.e., $R'-R'=\{r'-r\colon r',r\in R'\}\subset \Q$. However, looking into the proof there, the condition that the set $A$ has boundary of measure zero is used in  order to show that any  such tiling set $R$ must be \emph{periodic}, and the rest of the argument, \cite[Theorem~6 and Section~4]{Lagarias-Wang}, does not use this assumption. Thus, under the assumption that a bounded measurable set $A$ tiles the line by a \emph{periodic} set $R$, the argument of Lagarias--Wang gives the rationality of $R'=\mu(A)^{-1}R$.  Since any measurable tiling $F \oplus A =_\ae \T$ of the torus induces a periodic measurable tiling $(F + \Z) \oplus \tilde A =_\ae \R$ of the real line by a bounded measurable set $\tilde A$ of rational measure (defined as the image of $A$ under the identification of the circle $\T$ with the $[0,1)$), we conclude \cref{thm:torus}(iv).  In particular,  the conjecture from \cite{conley_grebik_pikhurko2020divisibility_of_spheres} may also be deduced from the results in  \cite{Lagarias-Wang}.

In fact, it was shown in \cite[Theorem~6.1]{Kolountzakis-Lagarias} that any tiling of $\R$ with a bounded measurable set $A$ is periodic.
Thus, combining this result with \cite[Section~4]{Lagarias-Wang} we have that every tiling of $\R$ by a bounded measurable set is periodic and rational.\footnote{We remark that classifying bounded measurable tiles $A\subset \R$ is a notoriously difficult problem even in the case when $A$ is a finite union of intervals. See, for instance, \cite{N77,tijdeman215decomposition,CM99,LS,LL21} and the references therein.}
\end{remark}

\begin{remark}
Using \cref{thm:torus}(iv), it is possible to fully describe all measurable tilings of the circle in terms of tilings of finite cyclic groups.
Namely, given $F=\{f_1,\dots, f_n\}\subseteq \mathbb{Q}$ and assuming that $f_1=0$, we find a $g\in \mathbb{T}^1\cap \mathbb{Q}$ such that $\langle g\rangle$, the cyclic group generated by $g$, contains $F'=\{f_1\mod 1,\dots, f_n\mod 1\}$ and is minimal with respect to set inclusion; for instance one can take $g = \frac{1}{q} \mod 1$ where $q$ is the least common multiple of the denominators of elements from $F$.
Let $\mathcal{T}_{g,F'}$ be the set of all $A'\subseteq \langle g\rangle$ such that $F'\oplus A'=\langle g\rangle$, i.e., $\mathcal{T}_{g,F'}$ consists of all tiles of the finite cyclic group $\langle g\rangle$ using translates $F'$.
Consider the action of $\langle g\rangle$ on $\mathbb{T}^1$ induced by $x\mapsto g+x$.
As $g\in \mathbb{Q}$, we infer that orbit of each $x\in \mathbb{T}^1$ is finite, in fact, of cardinality $|\langle g\rangle|$.
It follows that there is a measurable set $X\subseteq \mathbb{T}^1$ that intersects each orbit of the $\langle g\rangle$ in exactly one point.

There is a one-to-one correspondence between measurable sets $A\subseteq \mathbb{T}^1$ that satisfies $F\oplus A=_\ae \mathbb{T}^1$ and measurable functions
$$\psi:X\to \mathcal{T}_{g,F},$$
(where $\mathcal{T}_{g,F}$ is endowed with the discrete $\sigma$-algebra) that is given as follows: the measurable set $A_\psi$, that corresponds to $\psi$, is defined as $$A_\psi=\left\{y\in \mathbb{T}^1:\exists x\in X \ y\in \psi(x)+x\right \}.$$
Similarly, given a measurable tile $A$, the function
$$\psi_A(x)=\{h\in \langle g\rangle:h+x\in A\},$$
defined for $x\in X$, is measurable and $\psi_A(x)\in \mathcal{T}_{g,F'}$ almost surely.
\end{remark}

\subsection{The two-dimensional  case}\label{sec:twodim}

We now establish \cref{thm:torus}(ii).  By \cref{init}(i) we see that for each $f \in F \backslash \Q^2$ there exists a primitive $h_f \in \Z^2 \backslash \{0\}$ such that $h_f \cdot f \in \Q$.  Note that as $f \not \in \Q^2$, $h_f$ is determined up to sign.  The key observation (which is specific to two dimensions, as \cref{three-dim} shows) is

\begin{proposition}[All shifts are parallel]\label{parallel}  For any $f_1, f_2 \in F \backslash \Q^2$, one has $h_{f_1} = \pm h_{f_2}$.
\end{proposition}

We give two proofs of this proposition: a ``physical space'' proof inspired by the arguments in \cite{GreenfeldTao} which is based on the equidistribution theory of polynomials modulo one, and a ``Fourier analytic'' proof that exploits the fact that a non-trivial trigonometric polynomial can only vanish on a set of measure zero.

\begin{proof}[First proof]  Let $q$ be a natural number divisible by all the primes up to $|F|$, such that $qf \in \Z^2$ for all $f \in F\cap \Q^2$.  By \cref{structure-gen}(ii), we have a decomposition
\begin{equation}\label{onet2}
 \one_{\T^2} =_\ae \sum_{f \in F \cap \Q^2} \one_{f+A} + \sum_{f \in F \backslash \Q^2} \varphi_f
 \end{equation}
 where for each $f \in F \backslash \Q^2$, $\varphi_f \colon \T^2 \to [0,1]$ is $qf$-invariant and has mean $\mu(A)=1/|F|$.

For each $f \in F \backslash \Q^2$, some integer multiple $kqf$ of $f$ lies in the subtorus
$$ \langle h_f \rangle^\perp \coloneqq \{ x \in \T^2: h_f \cdot x = 0 \}.$$
Since $f \not \in \Q^2$, the translation action of $kqf$ is ergodic on this subtorus.  We conclude that $\varphi_f$ is in fact $\langle h_f \rangle^\perp$-invariant.  If we then define
$$ \Phi_1 \coloneqq \sum_{f \in F \backslash \Q^2: h_f = \pm h_{f_1}} \varphi_f$$
then $\Phi_1$ is also $\langle h_{f_1} \rangle^\perp$-invariant.  On reducing \eqref{onet2} modulo one, we have
\begin{equation}\label{ont}
0 =_\ae \Phi_1 + \sum_{f \in F_1} \varphi_f \hbox{ mod } 1,
\end{equation}
where $F_1$ consists of those $f \in F \backslash \Q^2$ with $h_f \neq \pm h_{f_1}$.
To eliminate the $\varphi_f$ terms we introduce the difference operators
$$ \partial_v g(x) \coloneqq g(x) - g(x-v)$$
for any $g \colon \T^2 \to \T$ and $v \in \T^2$.  Observe that these operators $\partial_v$ commute with each other.  If $f \in F_1$  we have
$$ \langle h_f \rangle^\perp + \langle h_{f_1} \rangle^\perp = \T^2$$
and hence if $g_f \in \T^2$  we may decompose $g_f = g'_f + g''_f$ where $g'_f \in \langle h_f \rangle^\perp$ and $g''_f \in \langle h_{f_1} \rangle^\perp$.  We then obtain the identities
$$ \partial_{g'_f} \Phi_1 =_\ae \partial_{g_f} \Phi_{1}$$
and
$$ \partial_{g'_f} \varphi_f =_\ae 0.$$
If one then applies each of the operators $\partial_{g'_f}$ in turn to \eqref{ont} for $f \in F_1$ to eliminate the $\varphi_f$ terms, we conclude that
$$ \left(\prod_{f \in F_1} \partial_{g_f}\right) \Phi_{1} =_\ae 0 \hbox{ mod } 1$$
whenever $g_f \in \T^2$.  Since $\Phi_{1}$ is $\langle h_{f_1} \rangle^\perp$-invariant, we may write
$$ \Phi_{1}(x) =_\ae \tilde \Phi_1(h_{f_1} \cdot x) \hbox{ mod } 1$$
for some measurable function $\tilde \Phi_1 \colon \T \to \T$, and then we have
\begin{equation}\label{pde}
    \partial_{\alpha_1} \dots \partial_{\alpha_{|F_1|}} \tilde \Phi_{1} =_\ae 0
\end{equation} 
for all $\alpha_1,\dots,\alpha_{|F_1|} \in \T$.  We claim that this implies that $\tilde \Phi_1$ is a linear function\footnote{See also \cite{Austin} for an extensive study on factorization of solutions to partial difference equations  in compact abelian groups (such as \eqref{pde}).}
\begin{equation}\label{nth}
\tilde \Phi_1(x) = n_1x + \theta_1 
\end{equation}
for some integer $n_1$ and $\theta_1 \in \T$, and almost all $x \in \T$.  We prove this by induction on the number $|F_1|$ of derivatives.  For $|F_1| \leq 1$, $\tilde \Phi_{1}$ is necessarily constant almost everywhere and the claim follows.  If $|F_1| > 1$, then by induction hypothesis, we see that for every $\alpha \in \T$ there exists an integer $n_\alpha$ and $\theta_\alpha \in \T$ such that
$$ \partial_\alpha \tilde \Phi_{1}(x) = n_\alpha x + \theta_\alpha$$
for almost every $x \in \T$.  As $\alpha \to 0$, the continuity of translation in the strong operator topology shows that $e^{2\pi i  \partial_\alpha \tilde \Phi_{1}}$ converges in $L^2(\T)$ to the constant $1$, and hence $n_\alpha$ must vanish for $\alpha$ sufficiently close to the origin.  Using the cocycle identity
$$ \partial_{\alpha+\beta} \tilde \Phi_{1}(x) = 
 \partial_{\alpha} \tilde \Phi_{1}(x) +  \partial_{\beta} \tilde \Phi_{1}(x-\alpha)$$
 and induction we conclude that $n_\alpha$ vanishes for all $\alpha$.  This argument also shows that the map $\alpha \mapsto \theta_\alpha$ is a continuous homomorphism from $\T$ to $\T$, and is thus of the form $\theta_\alpha = n_1\alpha$ for some integer $n_1$.  The function $\tilde \Phi_{1}(x) - n_1x$ is then almost everywhere constant (since all of its derivatives vanish almost everywhere), and the claim follows.  In particular we have
\begin{equation}\label{phi-f1}
\Phi_{1}(x) = n_1 h_{f_1} \cdot x + \theta_1 \hbox{ mod } 1
\end{equation}
for almost all $x \in \T^2$.
  
Now suppose for contradiction that $h_{f_2} \neq \pm h_{f_1}$.  If we set
$$ \Phi_{2} \coloneqq \sum_{f \in F \backslash \Q^2: h_f = \pm h_{f_2}} \varphi_f$$
then by repeating the previous arguments, we can find an integer $n_2$ and $\theta_2 \in \T$ such that
\begin{equation}\label{phi-f2}
\Phi_{2}(x) = n_2 h_{f_2} \cdot x + \theta_2 \hbox{ mod } 1
\end{equation}
for almost all $x \in \T^2$.

On the other hand, from \eqref{onet2} we have 
$$ \Phi_{1}(x) + \Phi_{2}(x) \leq 1 $$
for almost all $x \in \T^2$.  Since $\Phi_{1}$ is $\langle h_{f_1} \rangle^\perp$-invariant, and $\Phi_{2}$ is $\langle h_{f_2} \rangle^\perp$-invariant, and every coset of
$\langle h_{f_1} \rangle^\perp$ intersects every coset of $\langle h_{f_2} \rangle^\perp$, we have
$$ \| \Phi_{1} \|_{L^\infty(\T^2)} + \| \Phi_{2} \|_{L^\infty(\T^2)} = \| \Phi_{1} + \Phi_{2} \|_{L^\infty(\T^2)}  \leq 1.$$
We conclude that for some $i=1,2$ we have
$$ \| \Phi_{i} \|_{L^\infty(\T^2)} \leq 1/2.$$
Comparing this with \eqref{phi-f1} or \eqref{phi-f2} we conclude that $n_i$ must vanish, and $\Phi_{i}$ is equal almost everywhere to a constant $c_i$.  Since $\Phi_{i} \geq \varphi_{f_i}$ and $\varphi_{f_i}$ has mean $1/|F|$, we have $c_i \geq 1/|F|$. From \eqref{onet2} we then have the inequality
$$ \one_{\T^2} \geq \one_A + 1/|F|$$
almost everywhere, but this contradicts the positive measure of $A$.  Hence $h_{f_1} = \pm h_{f_2}$ as required.
\end{proof}

\begin{proof}[Second proof]  We introduce the set
$$ Q(A) \coloneqq \bigcup_{f \in F \cap \Q^2} f + A.$$
Since $Q(A)$ contains $A$ but not $f_1+A$, we have 
\begin{equation}\label{muq}
0 < \mu(Q(A)) < 1.
\end{equation}
On the one hand, from \eqref{onet2} and the arguments immediately following, we have a decomposition
\begin{equation}\label{1q1}
\one_{Q(A)} =_\ae \one_{\T^2} - \sum_{f \in F \backslash \Q^2} \varphi_f
\end{equation}
where each $\varphi_f$ is $\langle h_f\rangle^\perp$-invariant.  On the other hand, from \cref{init}(ii), we have a factorization
\begin{equation}\label{1q2}
\one_{Q(A)} =_\ae \prod_{f \in F \backslash \Q^2} \one_{A_f}
\end{equation}
where each $A_f$ is $qf$-invariant for some natural number $q$, and hence also $\langle h_f \rangle^\perp$-invariant by the arguments following \eqref{onet2}, and $0 < \mu(A_f) < 1$. To exploit these representations \eqref{1q1}, \eqref{1q2} we use the Fourier transform
$$ \hat F(k) \coloneqq \int_{\T^2} F(x) e^{-2\pi i k \cdot x}\ dx,\quad k\in\Z^2$$
defined for any $F \in L^2(\T^2)$. 

From \eqref{1q1} we see that the Fourier transform $\hat \one_{Q(A)} \in \ell^2(\Z^2)$ of $\one_{Q(A)}$ is supported on the set
$$ \bigcup_{f \in F \backslash \Q^2} \langle h_f \rangle$$
where $\langle h_f \rangle = \{ j h_f: j \in \Z \}$ is the group generated by $h_f$.
Indeed, this follows from the linearity of the Fourier transformation together with the fact that each $\hat \varphi_f$, for $f\in F\setminus \mathbb{Q}^2$ is supported on $\langle h_f \rangle $ as $\varphi_f$ is $\langle h_f\rangle^\perp$-invariant, i.e., $\varphi_f$ can only correlate with $e^{-2\pi i k \cdot x}$ where $k\in \langle h_f \rangle $, and similarly the support of $\hat \one_{\T^2}$ is $\{(0,0)\}\subset  \mathbb{Z}^2$. 
In particular, if $k \in \Z^2 \backslash \langle h_{f_1} \rangle$, then $\hat \one_{Q(A)}$ is only non-zero on finitely many elements of the coset $k + \langle h_{f_1} \rangle$. 
It follows that
$$G_{k,f_1}(x)\coloneqq\sum_{k' \in k + \langle h_{f_1} \rangle} \hat \one_{Q(A)}(k') e^{2\pi i k' \cdot x}$$
is a trigonometric polynomial.
We claim that $G_{k,f_1}$ agrees almost everywhere with the averaged function
$$ \tilde G_{k,f_1}(x) \coloneqq \int_{\langle h_{f_1} \rangle^\perp} \one_{Q(A)}(x+y) e^{-2\pi i k \cdot y}\ d\nu_{\langle h_{f_1} \rangle^\perp}(y)$$
(where $\nu_{\langle h_{f_1} \rangle^\perp}$ is Haar probability measure on $\langle h_{f_1} \rangle^\perp$).  Indeed, for $k' \in k + \langle h_{f_1} \rangle$ one computes
\begin{equation*}
    \begin{split}
        \int_{\T^2} \tilde G_{k,f_1}(x) e^{-2\pi i k' \cdot x} \ dx= & \ \int_{\T^2} \left(\int_{\langle h_{f_1} \rangle^\perp} \one_{Q(A)}(x+y) e^{-2\pi i k \cdot y}\ d\nu_{\langle h_{f_1} \rangle^\perp}(y)\right) e^{-2\pi i k' \cdot x} \ dx\\
        = & \  \int_{\langle h_{f_1} \rangle^\perp} \left(\int_{\T^2} \one_{Q(A)}(x+y) e^{-2\pi i k' \cdot (x+y)}\ dx\right) e^{-2\pi i (k-k') \cdot y} \ d\nu_{\langle h_{f_1} \rangle^\perp}(y)\\
        = & \ \int_{\langle h_{f_1} \rangle^\perp} \hat\one_{Q(A)}(k') e^{-2\pi i (k-k') \cdot y} \ d\nu_{\langle h_{f_1} \rangle^\perp}(y)\\
        = & \ \hat\one_{Q(A)}(k')
    \end{split}
\end{equation*}
by Fubini's Theorem and the fact that $(k-k')\cdot y=0$ for every $y\in \langle h_{f_1} \rangle^\perp$, and a similar computation shows that the Fourier coefficient $\int_{\T^2} \tilde G_{k,f_1}(x) e^{-2\pi i k' \cdot x} \ dx$ vanishes if $k' \not \in k + \langle h_{f_1} \rangle$.  By the Fourier inversion formula we conclude that $G_{k,f_1} =_\ae \tilde G_{k,f_1}$ as claimed.

On the other hand, from \eqref{1q2} and the $\langle h_{f_1}\rangle^\perp$-invariance of $A_{f_1}$, the function $G_{k,f_1} =_\ae \tilde G_{k,f_1}$ is supported on $A_{f_1}$ as
\begin{equation*}
    \begin{split}
        \tilde G_{k,f_1}(x)= & \ \int_{\langle h_{f_1} \rangle^\perp} \left(\prod_{f \in F \backslash \Q^2} \one_{A_f}\right) (x+y) e^{-2\pi i k \cdot y}\ d\nu_{\langle h_{f_1} \rangle^\perp}(y)\\
        = & \ \int_{\langle h_{f_1} \rangle^\perp} \one_{A_{f_1}}(x)\left(\prod_{f \in F \backslash \Q^2\cup \{f_1\}} \one_{A_f}\right) (x+y) e^{-2\pi i k \cdot y}\ d\nu_{\langle h_{f_1} \rangle^\perp}(y)\\
        = & \ 0
    \end{split}
\end{equation*}
whenever $\one_{A_{f_1}}(x)=0$.
Since $0 < \mu(A_{f_1}) < 1$, and a non-trivial trigonometric polynomial only vanishes on a set of measure zero, we conclude that $G_{k,f_1}$ vanishes whenever $k \in \Z^2 \backslash \langle h_{f_1} \rangle$.  By the Fourier inversion formula, this implies that $\hat \one_{Q(A)}$ is supported on $\langle h_{f_1} \rangle$.  If $h_{f_1} \neq \pm h_{f_2}$, then this argument shows that $\hat \one_{Q(A)}$ is supported on $\langle h_{f_1} \rangle \cap \langle h_{f_2} \rangle = \{0\}$, and hence $\one_{Q(A)}$ is constant.  But
this contradicts \eqref{muq}.  Hence we have $h_{f_1} = \pm h_{f_2}$ as claimed.
\end{proof}

With \cref{parallel} in hand, we can now complete the proof of \cref{thm:torus}(ii).  We may assume that $F \backslash \Q^2$ is non-empty, since the claim is trivial otherwise.  By \cref{parallel}, we may find an irreducible $h \in \Z^2 \backslash \{0\}$ such that $h \cdot f \in \Q$ for all $f \in F \backslash \Q^2$ (and hence for all $f \in F$).  By applying a suitable element of $\mathrm{SL}_2(\Z)$ (which does not affect the hypotheses or conclusions of \cref{thm:torus}(ii)), we may assume that $h = (0,1)$, thus we have
$$ F \subset \R \times \Q.$$
We can then cover $F$ by disjoint cosets $(\alpha_1,0) +\Q^2, \dots, (\alpha_k,0)+\Q^2$ for some real numbers $\alpha_1,\dots,\alpha_k$, whose differences are all irrational.  Of course we can assume that the intersections
$$ F_i \coloneqq F \cap ((\alpha_i,0) + \Q^2)$$
are non-empty for each $i=1,\dots,k$.

Let $f_i$ be an element of $F_i$.  Let $q$ be a natural number divisible by all primes less than or equal to $|F|$, such that $q(f-f_i) \in \Z^2$ for all $f \in F_i$.  By applying \cref{structure-gen}(ii) to the tiling $(F-f_i) \oplus A =_\ae \T^2$, we obtain a decomposition
\begin{equation}\label{1t2}
 \one_{\T^2} = \sum_{f \in F_i-f_i} \one_{f + A} + \sum_{f \in F \backslash F_i} \varphi_{f_i,f}
 \end{equation}
 where each $\varphi_{f_i,f} \colon \T^2 \to [0,1]$ is $q(f-f_i)$-invariant.  By construction, for each $f \in F \backslash F_i$, $q(f-f_i)$ lies in a coset $(\beta,0)+\Q^2$ for some irrational $\beta$, thus by the ergodic theorem $\varphi_{f_i,f}$ is invariant with respect to the action of $\R (1,0)$.  By \eqref{1t2}, we conclude that $\sum_{f \in F_i-f_i} \one_{f + A}$ is also $\R (1,0)$-invariant, and thus the set $F_i \oplus A$ is also $\R(1,0)$-invariant.
 
 If we now give each $f \in F_i$ the velocity $v_f \coloneqq (\alpha_i,0)$ for every $i=1,\dots,k$, then the (multi-)set $F^0 \coloneqq \{ f-v_f: f \in F \}$ lies in $\Q^2$, and if we then define the (multi-)set
 $$ F^t \coloneqq \{ f-v_f + tv_f: f \in F \}$$
 then we have by the $\R(1,0)$-invariance of $F_i \oplus A$ that
\begin{align*}
    \sum_{f \in F^t} \one_{f+A} &= \sum_{i=1}^k \sum_{f \in F_i} \one_{f-v_f+tv_f+A} \\
    &= \sum_{i=1}^k \one_{F_i \oplus A +((t-1)\alpha_i,0)}\\
    &=_\ae \sum_{i=1}^k \one_{F_i \oplus A}\\
    &= \sum_{f \in F} \one_{f+A}\\
    &=_\ae \one_{\T^2}
\end{align*}
and hence
$$ F^t \oplus A =_\ae \T^2.$$
Since $A$ has positive measure, this implies that the elements of $F^t$ are distinct (and so $F^t$ is a set, not just a multiset).  The claim in \cref{thm:torus}(ii) follows (with $v_0 = (1,0)$).

\begin{remark}  
Suppose that $F \oplus A =_\ae \T^d$ is a measurable tiling of the $d$-torus, $d\geq 2$, by some finite subset $F = \{f_1,\dots,f_n\}$ of $\R^d$ and some measurable subset $A$ of $\T^d$. Consider the velocities $v_i$ as in \cref{thm:torus}(i), and let  $F_{(v)} \coloneqq \{ f_i: v_i = v \}$ be  those elements of $F$ that are moving with velocity $v$.  \cref{thm:torus}(ii) gives that the sets $F_{(v)} \oplus A$ are $\R v$-invariant for every $v \in \R^2$.
In dimensions $d\ge 3$ this is no longer true. For instance, if $A=[0,1/2]^3$ and 
\begin{align*}
     F = \{ 
     &(0,0,0), (0,1/2,0),\\
     &(1/2, 0, 0), (1/2, 1/2, 0)\\
     &(\alpha,0 , 1/2), (\alpha,1/2, 1/2),\\
     &(\alpha+1/2, \beta, 1/2), (\alpha+1/2, \beta+1/2, 1/2)
     \}
\end{align*} for some irrationally numbers $\alpha,\beta$, and $F^0=\{0,1/2\}^3$. Then, if $v=(\alpha, \beta)$ or $v=(\alpha,0)$, we have that $F_{(v)}\oplus A$ is not $\R v$-invariant.

However, we do not know if there is any measurable tiling $F\oplus A=_\ae \T^d$, $d\ge 3$ that does not satisfy the following weaker analogue of \cref{thm:torus}(ii).
The velocities are replaced with piecewise linear functions $v_f:[0,1]\to \mathbb{R}^d$, for $f\in F$, such that $F^t\oplus A=_\ae \T^d$, where $F^t$ is defined analogously as in \cref{thm:torus}.
Moreover, if we set $F^t_{(v)}=\{f\in F:v'_f(t)=v\}$, for $v\in \R^d$, then $F^t_{(v)}\oplus A$ is $\R v$-invariant for every $v\in \R^d$ and $t\in [0,1]$ whenever all the derivatives exist.
In fact, we do not even know whether there exists any measurable tiling $F\oplus A=_\ae \T^d$, $d\ge 3$ such that there is no velocity $0\not= v\in \R^d$ for which there is a proper non-empty subset $F'$ of $F$ such that $F'\oplus A$ is $\R v$-invariant.
Note that the argument in \cite{Szabo} implies that if the tile $A$ is a cube then every tiling $F$ satisfies this weaker analogue of \cref{thm:torus}(ii).
On the other hand, if we are allowed to use more than one tile, then we  can construct such a tiling in $\T^3$, as follows.  Let $f \colon \T \to \R$ be the function
$$ f(x) \coloneqq \frac{\sin(\pi x)}{10}$$ 
and consider the following subsets $A_1, A_2, A_3$ of $\T^3$:
\begin{align*}
     A_1 &\coloneqq \{ (x,y,z \hbox{ mod 1}): (x,y) \in {\T}^2: f(x) \leq z < 1/3 + f(x+y) \}\\
 A_2 &\coloneqq \{ (x,y,z \hbox{ mod 1}): (x,y) \in {\T}^2: 1/3+f(x+y) \leq z < 2/3 + f(y) \}\\
 A_3 &\coloneqq \{ (x,y,z \hbox{ mod 1}): (x,y) \in {\T}^2: 2/3+f(y) \leq z < 1 + f(x) \}.
\end{align*}
It is a routine matter to verify that
$$ ((t,0,0) + A_1) \uplus ((0,t,0) + A_2) \uplus ((t,t,0) + A_3) =_\ae \T^3$$ 
for any $t \in \R$, but that none of the individual sets $A_1, A_2, A_3$ enjoy any translational symmetries.  Thus we see that there is a non-rigidity to the tiling problem
$$ (f_1+A_1) \uplus (f_2 + A_2) \uplus (f_3 + A_3) =_\ae \T^3$$
that cannot be explained purely by sliding each of the $A_1,A_2,A_3$ separately, or by translating the entire triple $A_1,A_2,A_3$ by a common shift.
\end{remark}

\subsection{The two-dimensional connected case}\label{sec:connected}

We now prove \cref{thm:torus}(iii).  The main new ingredient is the following classification of tilings of an interval by functions of connected support.

\begin{lemma}[Connected tilings of intervals]\label{connected-tiling}
Let $F$ be a finite multiset in $\R$, $[a,b]\subset\R$ be a finite interval and $\psi\colon \R\to [0,\infty)$ be a measurable function that is supported on a connected set. If
\begin{equation}\label{tiling}
    \one_F*\psi=_\ae \one_{[a,b]}
\end{equation} then there exists $m\in \N$, $c<c'$ such that $m\psi=_\ae \one_{[c,c']}$.
\end{lemma}

We remark that it is important here that the support of $\psi$ is connected, since the tiling
$$ \one_{\{0,1\}} * \one_{[0,1] \cup [2,3]} =_\ae \one_{[0,4]}$$
provides a counterexample in the disconnected case.
The proof of \cref{connected-tiling} follows from two observations sketched in \cref{fig:connected_tiles}. 

\begin{figure}
    \centering
    \includegraphics[width = \textwidth]{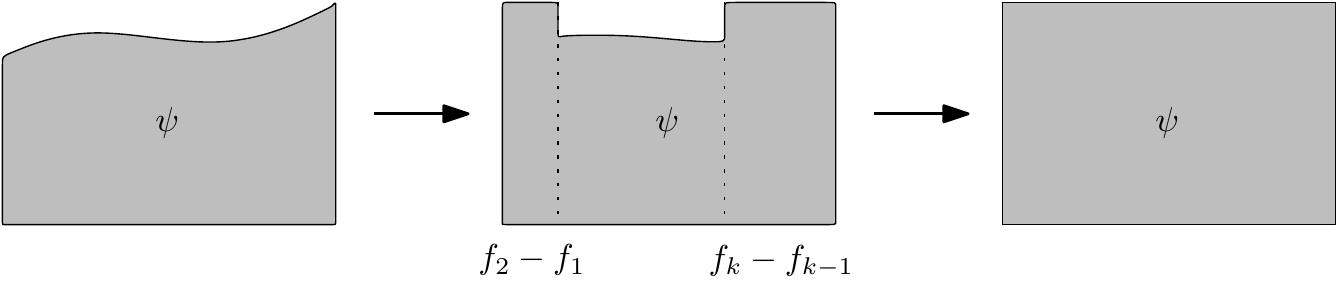}
    \caption{\cref{connected-tiling} follows from two observations. 
    First, we analyze the left and the right ``border'' of the tiling to conclude that almost everywhere on $[c, c + f_2 - f_1]$ and $[c' - (f_k - f_{k-1}), c']$, $\psi$ is equal to the same constant $1/m$ where $m$ is the multiplicity of both $f_1$ and $f_k$ in $F$. 
    We then consider the shift $f_{k-1}$ and since $\psi$ is assumed to be connected, we conclude that $c'-c \le f_k - f_{k-1}$ which in turn implies $\psi$ is constant on its whole support.   
    }
    \label{fig:connected_tiles}
\end{figure}

\begin{proof}
By translation and rescaling, we may normalize $[a,b] = [0,1]$ to be the unit interval, and also normalize $\min F = 0$.  
We enumerate the distinct elements of $F$ in order as 
$$0 = f_1<f_2<\dots<f_k$$ 
and write $$\one_F=\sum_{j=1}^k m_{j}\one_{\{f_j\}},$$ where $m_{j} \geq 1$ is the multiplicity of $f_j$ in $F$; thus
\begin{equation}\label{one-eq}
 \one_{[0,1]}(x) = \sum_{j=1}^k m_j \psi(x - f_j)
 \end{equation}
for almost every $x$. 

If $k=1$ then the claim immediately follows from \eqref{one-eq}.  Henceforth we assume $k > 1$.  Let $[c,c']$ be the support of $\psi$, then the support of $\sum_{j=1}^k m_j \psi(x - f_j)$ has infimum $c$ and supremum $f_k+c'$, thus by \eqref{one-eq} we have $c'>c=0$ and 
\begin{equation}\label{fk-range}
0 < f_k = 1-c' < 1.
\end{equation}

From \eqref{one-eq}, focusing attention in particular on the $j=1$ term on the right-hand side, we have
\begin{equation}\label{c1}
 1 \geq m_{1} \psi(x)
\end{equation}
for almost every $0 \leq x \leq 1$, with equality for almost every $0 \leq x \leq f_2$.  Focusing instead on the $j=k$ term, we have
\begin{equation}\label{ck}
 1 \geq m_k \psi(x-f_k)
\end{equation}
for almost every $0 \leq x \leq 1$, with equality for almost every $f_{k-1}+c' \leq x \leq 1$.  Combining these facts, we see that for almost every $f_k \leq x \leq \min(1, f_k+f_2)$ we have
$$ 1 \geq m_k \psi(x-f_k) = \frac{m_k}{m_1}$$
and for almost every $\max(0, f_{k-1}-f_k+c') \leq x \leq c'$ we have
$$ 1 \geq m_1 \psi(x) = \frac{m_1}{m_k}.$$
Since both of these ranges of $x$ have positive measure, we have $m_1=m_k=m$ for some natural number $m$, and
$$ \psi(x) = \frac{1}{m}$$
for almost every $0 \leq x \leq \delta \coloneqq \min(1,f_k+f_2)-f_k$.  Returning to \eqref{one-eq}, and isolating the $j=k$ term again, we see that
$$ 1 = 1 + \sum_{j=1}^{k-1} m_j \psi(x-f_j)$$
for almost all $f_k \leq x \leq \min(f_k+\delta,1)$.  In particular we have $\psi(x) = 0$ for almost all $f_k-f_{k-1} \leq x \leq \min(f_k+\delta,1) - f_{k-1}$.  Since $\psi$ is supported on $[0,c']$, this implies that $c' \leq f_k - f_{k-1}$.  Thus the functions $\psi(x-f_j)$ vanish for all $1 \leq j \leq k-1$ and almost all $x \geq f_k \geq f_j+c'$; inserting this back into \eqref{one-eq} we conclude that
$$ \one_{[0,1]}(x) = m \psi(x-f_k)$$
for almost all $x \geq f_k$.  Thus $m \psi =_\ae \one_{[0,1-f_k]}= \one_{[c,c']}$, and the claim follows.
\end{proof}

Now we can prove \cref{thm:torus}(iii).  Repeating the arguments from the preceding section, we may assume that $F \subset \R \times \Q$ is partitioned as $F_1 \cup \dots \cup F_k$, where each $F_i$ is non-empty, and of the form
$$ F_i \coloneqq F \cap ((\alpha_i,0) + \Q^2)$$
with each of the $F_i \oplus A$ being $\R (1,0)$-invariant.  

Suppose first that $k=1$, then $F=F_1$ is contained in a single coset of $\Q^2$, and hence lies in $\Q^2$ thanks to the normalization $0 \in F$.  In this case we can set all velocities equal to zero, and the claim follows.

Now suppose that $k > 1$. For each $i=1,\dots,k$, the $\R(1,0)$-invariant set $F_i \oplus A$ is equal almost everywhere to a set $\T \times I_i$, with the $I_i$ being of positive measure and partitioning $\T$. Since $A$ is open and connected, its projection to the vertical axis $\{0\} \times \T$ must then be an interval, and $I_i$ is (up to null sets) the union of finitely many translates of that interval.  In particular, each $I_i$ can be expressed as the disjoint union of finitely many intervals $I_{i,j} = [a_{i,j},b_{i,j}] \mod \Z$ in $\T$ for some $a_{i,j} < b_{i,j} < a_{i,j}+1$.
We can then partition $F_i$ into $F_{i,j}$ such that $F_{i,j} \oplus A =_\ae \T \times I_{i,j} $, or equivalently that
$$ \one_{\T \times I_{i,j}} =_\ae \one_{F_{i,j}} * \one_A.$$
On integrating out the horizontal variable we have
\begin{equation}\label{one-i}
 \one_{I_{i,j}} =_\ae \one_{\pi(F_{i,j})} * \psi
 \end{equation}
where $\pi \colon \R^2 \to \T$ is the projection homomorphism $\pi(x,y) \coloneqq y \hbox{ mod } 1$ (with $\pi(F_i)$ viewed as a multi-set), and $\psi \colon \T \to [0,+\infty)$ is the function
$$ \psi(y) \coloneqq \int_\T \one_A(x,y)\ dx.$$
Note that as $A$ is open and connected,  $\psi$ is supported on some interval supported inside some translate of $I_{i,j}$, and we can lift $\psi$ to a function $\tilde \psi\colon \R \to [0,+\infty)$ supported inside some translate of $[a_{i,j},b_{i,j}]$ so that one has a tiling
$$ \one_{[a_{i,j},b_{i,j}]} =_\ae \one_{\tilde F_{i,j}} * \tilde \psi$$
for some finite multiset $\tilde F_{i,j}$ in $\R$.  Applying \cref{connected-tiling}, we see that $\tilde \psi$ takes the form $\frac{1}{m} \one_{[c,c']}$ for some natural number $m$ and some interval $[c,c']$, which is contained in a translate of $[a_i,b_i]$ and thus has length strictly less than one; pushing back to $\T$, we conclude that $\psi = \frac{1}{m} \one_{[c,c'] \hbox{ mod } 1}$.  Inserting this into \eqref{one-i}, we see that the multi-set $\pi(F_{i,j})$ is in fact an arithmetic progression (of spacing $c'-c$) contained in a translate of $I_{i,j}$, with each element in this progression occurring with multiplicity $m$.  Thus one can partition each $F_{i,j}$ further into subsets $F_{i,j,k}$ of cardinality $m$, with each $\pi(F_{i,j,k})$ consisting of a single point $x_{i,j,k}$ with multiplicity $m$; in particular, each $F_{i,j,k} \hbox{ mod } \Z^2$ is contained in a single coset of $\R v_0 \hbox{ mod } \Z^2$, and
$$ \one_{x_{i,j,k} + [c,c']} =_\ae \one_{\pi(F_{i,j,k})} * \psi.$$
The right-hand side is the projection of $\one_{F_{i,j,k}} * \one_A$ after integrating out the horizontal variable.  Since this function is bounded  by $\one_F * \one_A =_\ae \one_{\T^2}$, we must therefore have
$$ \one_{\T  \times (x_{i,j} + [c,c'])} =_\ae \one_{F_{i,j,k}} * \one_A$$
or equivalently
$$ F_{i,j,k} \oplus A =_\ae \T \times (x_{i,j,k} + [c,c']).$$
In particular, each $F_{i,j,k} \oplus A$ is $\R v_0$-invariant.  This completes the proof.

\begin{remark}  The hypothesis that $A$ is open and connected can be relaxed to the hypothesis that $A$ is ``measurably connected'' in the following sense: for every $\xi \in \Z^2 \backslash \{0\}$, the function $\psi_\xi \colon \T \to [0,+\infty)$ defined by
$$ \psi_\xi(t) \coloneqq \int_{\xi \cdot x = t} \one_A,$$
where the integral is with respect to the Haar probability measure on $\{ x \in \T^2: \xi \cdot x=t\}$, has connected support (modulo null sets).  We leave the details of this generalization to the interested reader.
\end{remark}

\subsection{The high-dimensional case}\label{sec:highdim}

We now prove \cref{thm:torus}(i).  Let ${\mathcal T}(A) \subset (\T^d)^n$ denote the set of all tuples $(\tilde f_1, \dots, \tilde f_n) \in (\T^d)^n$ which generate a measurable tiling of the torus $\T^d$ in the sense that the translates $\tilde f_i + A$, $i=1,\dots,n$ partition $\T^d$ up to null sets, or equivalently that
$$ \sum_{i=1}^n \one_{\tilde f_i + A} =_\ae \one_{\T^d}.$$
Since translation is a continuous\footnote{This is, for instance, an immediate consequence of Lusin's theorem.} operation in the strong operator topology of (say) $L^2(\T^d)$, we see that the set ${\mathcal T}(A)$ is closed. By hypothesis,  this set ${\mathcal T}(A)$ is also non-empty; indeed, it contains the point
$$ \tilde f \coloneqq (f_1,\dots,f_n) \hbox{ mod } (\Z^d)^n.$$
Now let $q$ be the product of all the primes up to $n$.  By \cref{structure-gen}(i), we know that $(nq+1)F \oplus A =_\ae \T^d$ for all integers $n$, thus
$$ (nq+1) \tilde f \in {\mathcal T}(A).$$
We conclude that the orbit closure
$$ \overline{ \{ (nq+1) \tilde f: n \in \Z \} }$$
also lies in ${\mathcal T}(A)$.  We may write this orbit closure as
$$ \tilde f + H$$
where $H$ is the orbit closure
$$ H \coloneqq \overline{ \{ nq \tilde f: n \in \Z \} }.$$
Clearly, $H$ is a closed subgroup of the torus $(\T^d)^n$, and is thus a compact abelian Lie group.   By the classification of such groups (see e.g., \cite[Theorem 5.2(a)]{sepanski}), one can split $H = H^\circ \oplus K$, where $H^\circ$ is the identity connected component of $H$ (and thus a subtorus of $(\T^d)^n$) and $K$ is a finite subgroup of the torus $(\T^d)^n$.  In particular, $H$ is a finite union of rational cosets of $H^\circ$ (translates of $H^\circ$ by an element of $(\Q^d)^n$).
Since $q \tilde f \in H$, we conclude that $\tilde f + H$ is also a finite union of rational cosets of $H^\circ$.  In particular one has
$$ \tilde f \in \tilde f^0 + H^\circ \subset {\mathcal T}(A)$$
for some $\tilde f^0 \in (\Q^d)^n$.  One can write $H^\circ$ as ${\mathfrak h} \mod (\Z^d)^n$, where ${\mathfrak h} \leq (\R^d)^n$ is the Lie algebra of $H$ (or $H^\circ$).
Pulling back from $(\T^d)^n$ to $(\R^d)^n$, we conclude that
$$ (f_1,\dots,f_n) \in (f^0_1, \dots, f^0_n) + {\mathfrak h}$$
for some $(f^0_1, \dots, f^0_n) \in (\Q^d)^n$ with the property that
$$(f^0_1, \dots, f^0_n) + {\mathfrak h} \mod (\Z^d)^n \subset {\mathcal T}(A).$$ 
In particular, if we set
$$  (f^t_1, \dots, f^t_n) \coloneqq  (f^0_1, \dots, f^0_n) + t(v_1,\dots,v_n)$$
for $t \in \R$, where $v_i \coloneqq f_i - f^0_i$, then we have
$$ (f^t_1, \dots, f^t_n) \mod (\Z^d)^n \in {\mathcal T}(A)$$
and hence
$$ \{f^t_1, \dots, f^t_n\} \oplus A =_\ae \T^d.$$
The claims in \cref{thm:torus}(i) then follow.

\begin{remark}  The \emph{periodic tiling conjecture} \cite{GS-book, Lagarias-Wang} asserts that if one has a tiling $F \oplus A = \Z^d$ for some finite tile $F$ in $\Z^d$ and $A\subset \Z^d$, then there is also a periodic tiling $F \oplus A' = \Z^d$, thus $A'$ is a finite union of cosets of some lattice in $\Z^d$.  Currently this conjecture has only been established up to $d \leq 2$; see \cite{bhattacharya2020periodicity,GreenfeldTao}.  A variant of the argument used to prove \cref{thm:torus}(i) can establish the following partial result towards this conjecture: suppose that there is a homomorphism $T \colon \Z^d \to \T^m$ and a measurable subset $E$ of the torus $\T^m$ and a finite $F\subset \Z^d$ such that one has the measurable tiling $T(F) \oplus E =_\ae \T^m$, where an element $f$ of $\Z^d$ acts on $\T^m$ via translation by $T(f)$.  Then $F$ admits a periodic tiling $F \oplus A' = \Z^d$.  Indeed, if one defines ${\mathcal T}(F)$ to be the space of homomorphisms $\tilde T \colon \Z^d \to \T^m$ (which one can identify with $(\T^m)^d$) such that $\tilde T(F) \oplus E =_\ae \T^m$, then a variant of the above arguments shows that ${\mathcal T}(F)$ contains the orbit closure
$$ \overline{ \{ (nq+1) T: n \in \Z \} }.$$
By the above analysis, this orbit closure contains a rational point $T^0$, and by restricting the measurable tiling $T^0(F) \oplus E =_\ae \T^m$ to a generic coset of $T^0(\Z^d)$ and pulling back by $\Z^d$ one obtains a periodic tiling $F \oplus A' = \Z^d$; we leave the details to the interested reader.
\end{remark}

\bibliographystyle{alpha}
\bibliography{ref}

\appendix

\section{General structure theorem}\label{app}

In this appendix we establish

\begin{proposition}\label{sig}  \cref{structure-gen}(ii) continues to hold if the measure $\mu$ is assumed to be $\sigma$-finite rather than finite, and the action is assumed to be quasi-invariant rather than invariant, but the claim that the $\varphi_f$ has mean $\mu(A)$ is dropped.
\end{proposition}

In particular, we can recover \cref{structure-zd}(ii) which addresses the case of $\Z^d$ equipped with counting measure.

We begin with some easy reductions.
As every $\sigma$-finite measure can be replaced by a probability measure\footnote{For instance, if $\mu$ is a $\sigma$-finite measure and $X$ is exhausted by sets $K_n$ with $0 < \mu(K_n) < \infty$ then one can replace $\mu$ by the equivalent probability measure $\sum_{n=1}^\infty \frac{2^{-n}}{\mu(K_n)} \one_{K_n} \mu$.} in the same measure class, in particular, as the null sets are the same, we have that the action remains quasi-invariant, we may assume that $\mu$ is a probability measure.  Also we may assume without loss of generality that the measure space $X$ is complete, since otherwise one can pass to the completion and modify the $\varphi_f$ afterwards on a null set to recover measurability in the original $\sigma$-algebra.

The main new ingredient is that the use of the mean ergodic theorem is replaced by the use of a \emph{measurable medial mean}.
Recall that a medial mean $\mm$ is a linear functional from $\ell_\infty$ (the set of bounded sequences indexed by $\N$) to $\R$ which is positive, i.e., $\mm(\alpha)\ge 0$ whenever $\alpha \ge 0$, normalized, i.e., $\mm(1)=1$, and shift-invariant, i.e., $\mm(\alpha)=\mm(S(\alpha))$, where $S(\alpha)(n) \coloneqq \alpha(n+1)$.  We have the following key fact:

\begin{proposition}[Existence of measurable medial mean]\label{existenceMean} \cite[Section 3]{CieslaSabok-Hall}
Let $\nu$ be a Borel probability measure on $[0,1]^{\N}$.
Then there is a medial mean 
$$\operatorname{m} \colon \ell_\infty\to \R$$
that is $\nu$-measurable when restricted to $[0,1]^\N$ (i.e., it is measurable with respect to the completion of the Borel $\sigma$-algebra of $[0,1]^\N$ by $\nu$).
\end{proposition}

We have the following simple estimate:

\begin{claim}\label{cl:convmean}
Let $\operatorname{m}$ be a medial mean and $\alpha,\beta\in \ell_\infty$ such that $\lim_{N\to\infty} \alpha(N)-\beta(N)\to 0$.
Then $\operatorname{m}(\alpha)=\operatorname{m}(\beta)$.
\end{claim}
\begin{proof}
Let $\gamma(N)=\alpha(N)-\beta(N)$.
It is enough to show that $\operatorname{m}(\gamma)=0$.
We have $-|\gamma|\le \gamma\le |\gamma|$.
By positivity and additivity of $\operatorname{m}$ this gives $\operatorname{m}(|\gamma|)\ge \operatorname{m}(\gamma)\ge -\operatorname{m}(|\gamma|)$, as we have
$$\operatorname{m}(|\gamma|)-\operatorname{m}(\gamma)=\operatorname{m}(|\gamma|-\gamma)\ge 0$$
and similarly
$$\operatorname{m}(\gamma)-\operatorname{m}(-|\gamma|)=\operatorname{m}(\gamma-(-|\gamma|))\ge 0.$$
Consequently, it is enough to show that $\operatorname{m}(|\gamma|)=0$.
We have
$$\operatorname{m}(|\gamma|)=\operatorname{m}(S^k |\gamma|)\le \operatorname{m}(\max S^k |\gamma|)=\max S^k |\gamma|\to 0$$
by shift-invariance, normality and positivity.
\end{proof}

Now we are ready to prove \cref{sig}.  From the first part of \cref{structure-gen}(i) and \cref{quasi-rem} we have, as in the proof of \cref{structure-gen}(ii), that 
$$ \one_X = \sum_{f \in F} \frac{1}{N} \sum_{n=1}^N \one_{(f^{q})^n \cdot f \cdot A}$$
$\mu$-almost everywhere for all $N$. 
For each $f \in F$, let $\psi_f \colon X\to [0,1]^{\N}$ be the measurable function
$$\psi_f \coloneqq \left(\frac{1}{N} \sum_{n=1}^N \one_{(f^{q})^n \cdot f \cdot A}\right)_{N\in \N}.$$
Note that we have
$$\sum_{f\in F}\psi_f=(1,1,\dots)$$
$\mu$-almost everywhere.

Write $\nu_f$ for the push-forward of $\mu$ via $\psi_f$, where $f\in F$.
By \cref{existenceMean} (applied to $\nu=\frac{1}{|F|}\sum_{f\in F} \nu_f$), there is a medial mean $\mm$ that is simultaneously measurable for each $\nu_f$.
Define $\varphi_f \coloneqq \mm\circ \psi_f$.
By the definition of $\mm$, we have that $\varphi_f$ are positive measurable functions that satisfy
\begin{equation}
\one_X = \sum_{f \in F} \varphi_f.
\end{equation}
It is routine to verify that
$$|\psi_f(x)(N)-\psi_f(y)(N)|\le 2/N,$$
whenever $y=f^q\cdot x$, and that shows that $\varphi_f$ is $f^q$-invariant for every $f\in F$ by \cref{cl:convmean}.  Also from construction we have $\varphi_f =_\ae \one_{f \cdot A}$ if $f^q \cdot A =_\ae A$. \cref{sig} follows.

\section*{Funding}

JG was supported by Leverhulme Research Project Grant RPG-2018-424.
RG was partially supported by the Eric and Wendy Schmidt Postdoctoral Award and by NSF grant DMS-2242871. 
VR was supported by the European Research Council (ERC) under the European Unions Horizon 2020 research and
innovation programme (grant agreement No. 853109).
TT was partially supported by NSF grant DMS-1764034 and by a Simons Investigator Award. 

\section*{Acknowledgments}

We thank Nishant Chandgotia for  drawing our attention to  the reference \cite{leptin-muller}. JG and VR thank Víťa Kala and Oleg Pikhurko for insightful discussions. RG and TT thank Tim Austin for helpful conversations and suggestions. We are also grateful to the anonymous referees for several suggestions that improved the exposition of this paper.

\end{document}